\documentclass[11pt,a4paper]{article}

\usepackage[english]{babel}
\usepackage[utf8]{inputenc}

\usepackage{geometry}
\usepackage{amsmath,amsthm,amssymb,amsfonts}
\usepackage{mathtools}
\usepackage{bbm}
\usepackage{nicefrac} % for \nicefrac
\usepackage{atbegshi}

\usepackage{tikz}
\usepackage{ifdraft}

\usepackage{titlesec}

\usepackage[colorlinks=true,pdfstartview=FitV,linkcolor=blue,citecolor=blue,urlcolor=blue,final]{hyperref}

\usepackage[capitalize]{cleveref}
\crefname{enumi}{}{}
\crefname{equation}{}{}

\usepackage{xcolor}
\definecolor{revblue}{RGB}{0,80,180}

\numberwithin{equation}{section}
\theoremstyle{plain}
\newtheorem{theorem}{Theorem}[section]
\newtheorem{corollary}[theorem]{Corollary}
\newtheorem{proposition}[theorem]{Proposition}
\newtheorem{lemma}[theorem]{Lemma}
\newtheorem{remark}[theorem]{Remark}

\def \R {\mathbb{R}}
\def \N {\mathbb{N}}

\newcommand{\dx}{\,{\rm d}x}
\newcommand{\dy}{\,{\rm d}y}
\newcommand{\dt}{\,{\rm d}t}
\newcommand{\ds}{\,{\rm d}s}

\renewcommand{\phi}{\varphi}

\newcommand{\ud}{\textrm{d}}

\def\Xint#1{\mathchoice
{\XXint\displaystyle\textstyle{#1}}%
{\XXint\textstyle\scriptstyle{#1}}%
{\XXint\scriptstyle\scriptscriptstyle{#1}}%
{\XXint\scriptscriptstyle\scriptscriptstyle{#1}}%
\!\int}
\def\XXint#1#2#3{{\setbox0=\hbox{$#1{#2#3}{\int}$}
\vcenter{\hbox{$#2#3$}}\kern-.5\wd0}}

\def\fint{\Xint-}

\newcommand{\MR}[1]{\href{http://www.ams.org/mathscinet-getitem?mr=#1}{MR#1}}

\newcommand{\addperiod}[1]{#1.}

\titleformat{\section}
	{\centering\normalfont\Large}{\thesection.}{0.5em}{}

\titleformat{\subsection}[runin]
	{\normalfont\bfseries}{\thesubsection.}{0.5em}{\addperiod}

\titleformat{\subsubsection}[runin]
	{\normalfont\bfseries}{\thesubsubsection.}{0.5em}{\addperiod}

\titleformat*{\subsubsection}{\normalfont\itshape}
\titleformat*{\paragraph}{\bfseries}
\titleformat*{\subparagraph}{\large\bfseries}

\title{$\chi^2$--cut-off phenomenon for Galerkin projections of Fokker--Planck equations with monomial potentials}
\author{Benny Avelin
\thanks{Department of Mathematics, Uppsala University, Sweden.
{\footnotesize \href{mailto:benny.avelin@math.uu.se}{benny.avelin@math.uu.se}.}}
\and
Gerardo Barrera
\thanks{Center for Mathematical Analysis, Geometry and Dynamical Systems, Mathematics Department,
Instituto Superior T\'ecnico, Universidade de Lisboa, Portugal.
{\footnotesize \href{mailto:gerardo.barrera.vargas@tecnico.ulisboa.pt}{gerardo.barrera.vargas@tecnico.ulisboa.pt}.}}
}

\begin{document}
\AtBeginShipoutFirst
{
\begin{tikzpicture}[remember picture, overlay]
\node[anchor=north west, xshift=1.5cm, yshift=-1.5cm] at (current page.north west){
\includegraphics[width=4cm]{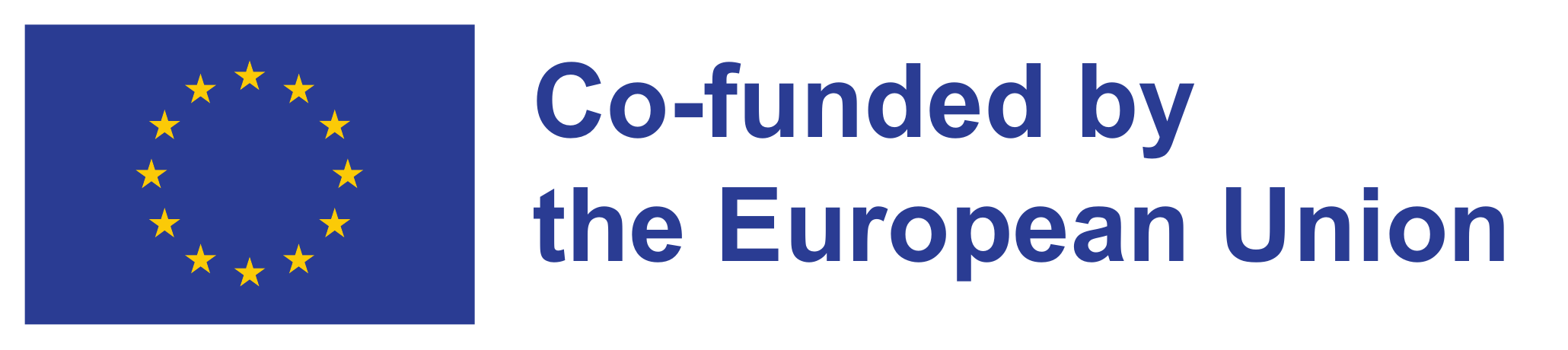}  %Change filename, size or position as needed
};
\end{tikzpicture}
}

\maketitle
\begingroup
\renewcommand{\thefootnote}{\ifcase\value{footnote}\or*\or**\fi}

\endgroup

\begin{abstract}
	In this manuscript, we  establish the existence/non-existence of the cut-off phenomenon for the Langevin--Kolmogorov random dynamics with monomial convex potentials, possible singular, and driven by a Brownian motion with small strength. We consider a  truncated $\chi^2$-distance, that is, a distance based on Galerkin projections of the eigensystem, and show that not only a refined knowledge of the eigenvalues is needed but also a refined asymptotics of the growth for the eigenfunctions of the Fokker--Planck equations associated to the Langevin--Kolmogorov dynamics. In addition, this explicit analysis yields asymptotics of the mixing times and, in some regimes, information on the limiting profile, going beyond the product condition and the cut-off window alone.
\end{abstract}

\setcounter{tocdepth}{2}
\tableofcontents

\section{{Introduction and main results}}\label{sec:main}

The study of the long-time behavior of time-evolving systems out of equilibrium (a.k.a.~ergodicity) and the quantification of the convergence and mixing to equilibrium is one of the most important problems in theoretical mathematics, modern probability and statistical physics with a vast literature, see for instance~\cite{BGL,Bogachev, Bolley,Ji,KU,booklelievreetalt,pages,Vi2009} and the references therein.

In what follows, we always consider a parameter $\varepsilon\in (0,1]$, which regulates the magnitude of the noise.
In this manuscript, we study the convergence to equilibrium of the so-called Langevin--Kolmogorov process for monomial convex potentials and driven by a Brownian motion with small noise intensity $\varepsilon$. The one-dimensional overdamped  Langevin--Kolmogorov process $(X^{(\varepsilon)}_t)_{t\geq 0}$ is defined as the unique strong solution of the following stochastic differential equation (SDE for shorthand)
\begin{equation} \label{eq:Kolmogorov}
	\left\{
	\begin{array}{r@{\;=\;}l}
		\ud X^{(\varepsilon)}_t & -V^{\prime}(X^{(\varepsilon)}_t) \dt+\sqrt{2\varepsilon}\ud B_t
		\quad\textrm{ for } \quad  t\geq  0,                                                      \\
		X^{(\varepsilon)}_0     & x_0,
	\end{array}
	\right.
\end{equation}
where $V:\R \to \R$ is a $\mathcal{C}^1$ confining potential, $x_0\in \mathbb{R}$ is a deterministic initial datum and $(B_t)_{t\geq 0}$ is a one-dimensional standard Brownian motion.
Since $V$ is convex, the SDE \cref{eq:Kolmogorov} has a unique strong solution, see~\cite[Theorem~3.5]{Maobook} or~\cite[Theorem~10.2.2]{SVbook}.
We note that the generator of the Langevin--Kolmogorov process associated
to~\cref{eq:Kolmogorov} acting on $u\in \mathcal{C}^2(\mathbb{R},\mathbb{R})$
is given by
\begin{equation}\label{eq:FP}
	\mathcal{L}_\varepsilon u:= \varepsilon u'' - V' u'= \varepsilon e^{V/\varepsilon } (e^{-V/\varepsilon } u')',
\end{equation}
and the associated Fokker--Planck operator $\mathcal{L}^\ast_\varepsilon$, which is the formal adjoint of $\mathcal{L}_\varepsilon$, is given by
\begin{equation}\label{eq:OU:fokker}
	\mathcal{L}_\varepsilon^\ast \varrho:=  ( \varrho V'+ \varepsilon \varrho')',
	\quad \textrm{ where }\quad \varrho\in \mathcal{C}^2(\mathbb{R},[0,\infty)),
\end{equation}
where $f'$ denotes the derivative of $f\in \mathcal{C}^2(\mathbb{R},[0,\infty))$ in the spatial variable $x\in \mathbb{R}$.
Furthermore, $\mathcal{C}^2$ densities of the marginals for~\cref{eq:Kolmogorov}, $\varrho_{\varepsilon,t}(x):=\varrho^{(\varepsilon)}(x,t)$, solve the partial differential equation $\partial_t \varrho = \mathcal{L}^*_\varepsilon \varrho$.
For further details, we refer to~\cite[Chapter~1]{Bogachev}.
When the operator \cref{eq:FP} has a spectral gap and  the potential $V$ satisfies the growth condition:
\begin{equation} \label{eq:spectralcondition}
	\frac{1}{2} |V^{\prime}(x)|^2 - V^{\prime\prime}(x) \to \infty \quad \textrm{ as } \quad |x| \to \infty
\end{equation}
then the operator has a discrete spectrum
\begin{equation}\label{ec:eignsys}
	\lambda_{0,\varepsilon} < \lambda_{1,\varepsilon} < \lambda_{2,\varepsilon} < \ldots \quad
	\textrm{ with corresponding eigen-basis }\quad \psi_{0,\varepsilon},\psi_{1,\varepsilon},\psi_{2,\varepsilon},\ldots
\end{equation}
for the Lebesgue space $L^2(C_\varepsilon e^{-V(x)/\varepsilon} \ud x)$, where $C_\varepsilon$ is a normalization constant, see for instance~\cite{Jac} or~\cite[Chapter~4]{BGL}. Without loss of generality, we can assume that $\lambda_{0,\varepsilon}= 0$ and $\psi_{0,\varepsilon}= 1$. In addition, we assume that $\{\psi_{k,\varepsilon}\}_{k\geq 0}$ is \textit{normalized to be orthonormal} in $L^2(C_\varepsilon e^{-V(x)/\varepsilon} \ud x)$.
For convenience when $\varepsilon=1$ we drop the second variable in the notation of the eigensystem given in \cref{ec:eignsys}, and we write
\begin{equation*}
	0=\lambda_{0} < \lambda_{1} < \lambda_{2} < \ldots \quad
	\textrm{ with corresponding eigen-basis }\quad \psi_{0}=1,\psi_{1},\psi_{2},\ldots
\end{equation*}
Our goal is the study of the so-called cut-off phenomenon for the convergence to equilibrium for the
one-dimensional over-damped Langevin--Kolmogorov process $(X^{(\varepsilon)}_t)_{t\geq 0}$
out of equilibrium when the potential $V$ has the form
\begin{equation}\label{def:V}
	V(x):=K_\gamma|x|^{\gamma+1},\quad  x\in \mathbb{R}\quad \textrm{with exponent}\quad \gamma> 0,
\end{equation}
where $K_\gamma$ is a positive constant. For simplicity and without loss of generality, along this manuscript we assume that $K_\gamma:=\frac{1}{\gamma+1}$.
We point out that the exponent $\gamma=0$ is critical, due to the lack of a discrete spectrum and only the essential spectrum exists, see~\cite[Chapter 4]{BGL} for further details.

For $\gamma>0$, we consider the process \cref{eq:Kolmogorov} with the starting condition $X^{(\varepsilon)}_0$ having distribution
\begin{equation}\label{eq:initialcond}
	\varrho_{\varepsilon,0} = \frac{\mathbbm{1}_{B_\delta(x_0)}}{|B_\delta(x_0)|},
\end{equation}
where $x_0\in \mathbb{R}$, $\delta:= \delta(\varepsilon) > 0$ will be chosen properly, $B_{\delta}(x_0)$ denotes the open interval of radius $\delta$ around $x_0$, $\mathbbm{1}_A$ is the indicator function of a given set $A$, and $|\cdot|$ denotes the length of $B_{\delta}(x_0)$, i.e., $|B_\delta(x_0)|=2\delta$. Clearly,  $\varrho_{\varepsilon,0}$ is bounded.

Under the aforementioned assumptions,
the process \cref{eq:Kolmogorov} is uniquely ergodic and
the density of its invariant Boltzmann probability  measure  is explicitly given by
\begin{equation}\label{ec:invm}
	\varrho_{\varepsilon,\infty}(x)= C_\varepsilon e^{-V(x)/\varepsilon}\quad \textrm{ for }\quad x\in \mathbb{R},
\end{equation}
where
$C_\varepsilon$ is a normalization constant ($C^{-1}_\varepsilon$ is the so-called partition function usually denoted by $\mathcal{Z}_\varepsilon$), see~\cite[Chapter 2]{Royer2007}.

We now consider the $\chi^2$-distance (a.k.a.~divergence measure) to equilibrium, that is to say,
\begin{equation*}
	\begin{split}
		d_\varepsilon (t,x_0): & = \left ( \int_{\mathbb{R}} |\varrho_{\varepsilon,t}(x) - \varrho_{\varepsilon,\infty}(x)|^2 \varrho^{-1}_{\varepsilon,\infty}(x) \dx\right )^{1/2}                                                                   \\
		                       & =\left ( \int_{\mathbb{R}} \left|\frac{\varrho_{\varepsilon,t}(x)}{\varrho_{\varepsilon,\infty}(x)} - 1\right|^2 \varrho_{\varepsilon,\infty}(x) \dx\right )^{1/2}\in [0,\infty)\quad \textrm{for each}\quad t\geq 0,
	\end{split}
\end{equation*}
which with the help of \cref{lem:separation} can be written as
\begin{equation*}
	d_\varepsilon(t,x_0) = \left(\sum_{k=1}^\infty \left ( \int_{\mathbb{R}} \varrho_{\varepsilon,t}(x) \psi_{k,\varepsilon}(x) \dx \right )^2\right)^{1/2},\quad t\geq 0.
\end{equation*}
We consider instead a truncated version of the $\chi^2$-distance to equilibrium, retaining only the first $n$ modes. This captures the dominant behavior and yields explicit, computable profiles; it is also the natural object for Galerkin-based numerical methods for the Fokker--Planck equation. More precisely,
for each $n\in \mathbb{N}:=\{0,1,2,\ldots,\}$
we consider the Galerkin projection
$\Pi_{n,\varepsilon}$ to be the spectral projector onto the space spanned by the first $(n+1)$-eigenfunctions of \cref{eq:FP}, that is,
\[
	\Pi_{n,\varepsilon} f := \sum_{k=0}^n \langle f,\psi_{k,\varepsilon}\rangle_{\varrho_{\varepsilon,\infty}} \psi_{k,\varepsilon}.
\]
We then define
the truncated $\chi^2$-distance  to equilibrium in that spectral subspace as follows:
\begin{equation}\label{eq:Ga}
	d_{n,\varepsilon} (t,x_0):=\left ( \int_{\mathbb{R}} \left \lvert \Pi_{n,\varepsilon}\left [\frac{\varrho_{\varepsilon,t}}{\varrho_{\varepsilon,\infty}} \right ](x) - 1 \right \lvert^2 \varrho_{\varepsilon,\infty}(x) \dx \right )^{1/2},\quad t\geq 0.
\end{equation}
When the initial point $x_0$ is fixed and clear from the context, we occasionally write $d_{n,\varepsilon}(t)$ instead of $d_{n,\varepsilon}(t,x_0)$; similarly for $d_\varepsilon(t)$.
By \cref{cor:separation} one can see that
\begin{equation} \label{eq:dnvarepsilon_intro}
	d_{n,\varepsilon} (t,x_0)= \left(\sum_{k=1}^n e^{- 2 \lambda_{k,\varepsilon} t} \left ( \int_{\mathbb{R}} \varrho_{\varepsilon,0}(x) \psi_{k,\varepsilon}(x) \dx \right )^2\right)^{1/2},\quad t\geq 0.
\end{equation}
Note that the map $t\mapsto d_{n,\varepsilon} (t,x_0) $ is continuous and strictly decreasing as $t $ increases. Moreover, for any $n\in \mathbb{N}$, $\varepsilon>0$ and $x_0\in \mathbb{R}$ fixed we have
$d_{n,\varepsilon} (t,x_0)\to 0$, as $t\to \infty$.
Hence, the so-called $\eta$-mixing time is well-defined and unique, that is, given a prescribed error $\eta>0$,
\begin{equation*}
	\tau_{\mathsf{mix}}(\eta;\varepsilon,x_0,n):=\inf\{t\geq 0: d_{n,\varepsilon} (t,x_0)\leq \eta\}.
\end{equation*}

The main results of this manuscript concern the so-called cut-off phenomenon for the Galerkin projection distance given in \cref{eq:Ga}.
The cut-off phenomenon was coined by D.~Aldous and P.~Diaconis
in the study of the convergence to equilibrium in the total variation distance of Markov chains models of card-shuffling, the Ehrenfest urn and random transpositions, see~\cite{AD,DI}. Roughly speaking, the distance of the current state to its equilibrium behaves as an approximate step function that remains
close to its maximal value for a while, and then suddenly falls to zero as the running time parameter reaches a critical
threshold, which corresponds to the so-called mixing time.
In recent times, the cut-off phenomenon has been studied from others distances, and it has been shown to appear in diverse fields such as: coagulation-fragmentation systems~\cite{Murray}, deep limits of neural networks~\cite{Avelin}, quantum Markov chains~\cite{Kastoryano12}, open quadratic fermionic systems~\cite{Vernier20}, stochastic simulations of chemical kinetics~\cite{BOK10}, viscous energy shell-model of turbulence~\cite{BHPPshell}, and random quantum circuits~\cite{Oh}.

In the sequel, we rigorously define  cut-off, window cut-off and profile cut-off for \cref{eq:Ga}.
Recall that the dimension of the Galerkin projector $\Pi_{n,\varepsilon}$ is fixed, that is, $n\in \mathbb{N}$ is fixed, and $\varepsilon>0$ is the complexity parameter.
We say that \cref{eq:Ga} exhibits \textit{cut-off} iff there exists a cut-off time $t_\varepsilon>0$ (it is not required that $t_\varepsilon\to \infty$ as $\varepsilon\to 0$ and in some cases $t_\varepsilon\to 0$ as $\varepsilon\to 0$) such that
\begin{equation}\label{eq:cutdef}
	\lim\limits_{\varepsilon\to 0}d_{n,\varepsilon}(ct_\varepsilon,x_0)=\begin{cases}
		\infty & \textrm{ for }\quad c\in (0,1),        \\
		0      & \textrm{ for } \quad  c\in (1,\infty).
	\end{cases}
\end{equation}
We remark that the usual definition of cut-off phenomenon the upper limit in \cref{eq:cutdef} is required to be equal to 1 instead of $\infty$, see for instance~\cite{AD,BBF1,BY,DI,Levin}. However, since we are working with the $\chi^2$-distance, which is unbounded, we adapt the definition accordingly.
We say that \cref{eq:Ga} exhibits \textit{window cut-off} iff there exist a cut-off time  $t_\varepsilon>0$ and a time-window $w_\varepsilon>0$ such that $w_\varepsilon/t_\varepsilon
	\to 0$ as $\varepsilon\to 0$, and for any $r\in \mathbb{R}$ we have
\begin{equation}\label{eq:def:wcut}
	\begin{split}
		\limsup\limits_{\varepsilon\to 0}d_{n,\varepsilon}(t_\varepsilon+rw_\varepsilon,x_0) & =:\mathfrak{p}_+(r), \\
		\liminf\limits_{\varepsilon\to 0}d_{n,\varepsilon}(t_\varepsilon+rw_\varepsilon,x_0) & =:\mathfrak{p}_-(r),
	\end{split}
\end{equation}
satisfying $\mathfrak{p}_-(-\infty)=\infty$ and $\mathfrak{p}_+(\infty)=0$. If additionally, $\mathfrak{p}(r):=\mathfrak{p}_+(r)=\mathfrak{p}_{-}(r)$ for all $r\in \mathbb{R}$, that is,
\begin{equation}\label{def:pro}
	\lim\limits_{\varepsilon\to 0}d_{n,\varepsilon}(t_\varepsilon+rw_\varepsilon,x_0) =\mathfrak{p}(r)\quad \textrm{ for all }\quad r\in \mathbb{R},
\end{equation}
we say that \cref{eq:Ga} exhibits \textit{profile cut-off} with profile function $\mathfrak{p}$.
We point out that $t_{\varepsilon}$, $w_\varepsilon$ and $\mathfrak{p}$ depend on the initial condition $x_0$ and on the dimension of the Galerkin projector $n$.
If additionally, the profile function $\mathfrak{p}$ is strictly decreasing and continuous  we have for any $\eta \in (0,\infty)$ that
\begin{equation*}
	\tau_{\mathsf{mix}}(\eta;\varepsilon,x_0,n)=t_\varepsilon+\mathfrak{p}^{-1}(\eta)w_\varepsilon+\textrm{o}(w_\varepsilon)\quad \varepsilon\to 0,
\end{equation*}
where $\mathfrak{p}^{-1}$ is the inverse function of $
	\mathfrak{p}$, see Lemma E.1 in~\cite{BARESQ}.
\begin{remark}[Cut-off phenomenon: window and profile]
	\hfill

	\noindent
	The upper and lower limits
	in~\cref{eq:def:wcut} is commonly referred to as window cut-off in the context of Markov processes when the distance $d_{n,\varepsilon}$ is replaced by the total variation distance and the complexity parameter $\varepsilon>0$ is the cardinality of the state space.
	Also, the non-existence of the timescale satisfying~\cref{eq:cutdef} is commonly referred to as there being no cut-off phenomenon, see for instance~\cite{AD,DI} and~\cite[Chapter~18]{Levin}.
\end{remark}

In \cref{thm:main,thm:main:sub,thm:main:2} below we assume that $V$ is the potential defined in \cref{def:V}.
We recall that $\lambda_{k,\varepsilon}$ is the $k$-th non-zero eigenvalue of the Fokker--Planck operator defined in \cref{eq:FP}, and $\lambda_k$ is the $k$-th non-zero eigenvalue with $\varepsilon=1$.

The first main result of this manuscript is the following, which establishes a window cut-off phenomenon for $\gamma\in(1/3,1)$.

\begin{theorem}[Window cut-off phenomenon for Galerkin projection, $\gamma\in (1/3,1)$]\label{thm:main}
	\hfill

	\noindent
	Assume that $\gamma \in (1/3,1)$, $\delta(\varepsilon) = \varepsilon^{\frac{1-\gamma}{1+\gamma}}$, and $n\in \mathbb{N}$.

	\noindent
	For any $x_0\neq 0$  window cut-off holds true in the sense of~\cref{eq:def:wcut},
	where $d_{n,\varepsilon}$ is defined in \cref{eq:Ga}, $t_\varepsilon$
	is defined as the unique solution to
	\begin{equation}\label{eq:defThat}
		d_{1,\varepsilon} ( t_\varepsilon,x_0) = 1,
	\end{equation}
	in other words,
	\begin{equation*}
		t_\varepsilon = \frac{1}{\lambda_{1,\varepsilon}} \log \left  ( \left \lvert \int \varrho_{\varepsilon,0}(x) \psi_{1,\varepsilon}(x) \dx  \right \rvert \right ),
	\end{equation*}
	and the time window is given by $w_\varepsilon:= \varepsilon^{\frac{1-\gamma}{1+\gamma}}$.

	\noindent
	For $x_0 = 0$, let
	\begin{equation*}
		E_n:=\{1\leq j\leq n:\ \psi_j(0)\neq 0\}.
	\end{equation*}
	If $E_n=\emptyset$, then $d_{n,\varepsilon}(t,0)=0$ for all $t\geq 0$. Otherwise, letting $m:=\max E_n$, the centered initial datum exhibits window cut-off (indeed profile cut-off) with cut-off time determined by the condition
	\begin{equation*}
		e^{-2\lambda_{m,\varepsilon} t_\varepsilon}\left(\int \varrho_{\varepsilon,0}(x)\psi_{m,\varepsilon}(x)\,\dx\right)^2=1,
	\end{equation*}
	the same time window $w_\varepsilon:=\varepsilon^{\frac{1-\gamma}{1+\gamma}}$, and profile function $\mathfrak p(r)=e^{-\lambda_m r}$.
\end{theorem}

The proof  is given in \cref{sec:proofuno}.

\begin{remark}[The \emph{highest} mode $m=\max E_n$ controls the cut-off at $x_0=0$]\hfill

\noindent
	Intuitively, the slowest mode (smallest $\lambda_1$) should dominate, since all other terms decay faster for fixed $\varepsilon$.  This reasoning is valid at fixed $\varepsilon$ but fails in the small-noise limit $\varepsilon\to 0$; by \cref{lem:pointwise_intro}, for $x_0=0$ and any even mode $j$ the Fourier coefficient satisfies
	\[
	  |c_{j,\varepsilon}| \approx \varepsilon^{\,\gamma(1-\gamma)/(1+\gamma)}
	  \exp\!\Bigl(\tfrac{\lambda_j}{1-\gamma}\,\varepsilon^{\gamma(\gamma-1)/(\gamma+1)}\Bigr).
	\]
	Since $\gamma(\gamma-1)/(\gamma+1)<0$ the exponential blows up as $\varepsilon\to 0$, and faster for larger $\lambda_j$.  At the cut-off time $t_\varepsilon$ defined by normalizing the $m=\max E_n$ mode, every lower mode $j\in E_n$ with $j<m$ satisfies $e^{-2\lambda_{j,\varepsilon}t_\varepsilon}|c_{j,\varepsilon}|^2\to 0$, since the ratio $\lambda_j/\lambda_m<1$ produces a net decay.  If instead $m'=\min E_n$ were used, higher modes would blow up to $+\infty$ at $t_\varepsilon$, which cannot define a valid profile.
\end{remark}

\begin{remark}[Explicit asymptotic bounds for the cut-off time for $\gamma\in (1/3,1)$]\hfill
	\label{rem:timecutoff}

	\noindent
	In order to explicitly compute $t_\varepsilon$ in \cref{eq:defThat} we need to know $\psi_{1,\varepsilon}$, which in general is a difficult task. For $\gamma\in (1/3,1)$ \cref{lem:time_intro} below yields for small enough $\varepsilon$ that there exists a constant $C > 1$ depending on $\gamma$ such that
	\begin{equation*}
		\frac{\log(1/C)}{\lambda_{1}} \varepsilon^{\frac{1-\gamma}{1+\gamma}}+ \frac{|x_0|^{1-\gamma}}{1-\gamma}
		\leq
		t_\varepsilon
		\leq
		\frac{\log(C)}{\lambda_{1} } \varepsilon^{\frac{1-\gamma}{1+\gamma}} + \frac{|x_0|^{1-\gamma}}{1-\gamma},
	\end{equation*}
	where $\lambda_1$ is the first non-zero eigenvalue of the Fokker--Planck operator defined in \cref{eq:FP} with $\varepsilon=1$. In particular, $t_\varepsilon=\mathsf{O}(1)$, as $\varepsilon\to 0$.
\end{remark}

The second main result of this manuscript establishes the occurrence of a profile cut-off phenomenon for all 
$\gamma \in (0,1/3]$.
\begin{theorem}[Profile cut-off phenomenon for Galerkin projection,
		$\gamma\in (0,1/3\rbrack$]\label{thm:main:sub}
	\hfill

	\noindent
	Assume that $\gamma \in (0,1/3]$, $\delta(\varepsilon) = \varepsilon^{\frac{1-\gamma}{1+\gamma}}$, and $n\in \mathbb{N}$.

	\noindent
	For any $x_0\neq 0$  profile cut-off holds true in the sense of~\cref{def:pro},
	where $d_{n,\varepsilon}$ is defined in \cref{eq:Ga}, and for sufficiently small $\varepsilon>0$ the cut-off time $t_\varepsilon$
	is the unique solution to
	\begin{equation}\label{eq:defThatnuev}
		e^{-2\lambda_{n,\varepsilon} t_\varepsilon}
		\left ( \int \varrho_{\varepsilon,0}(x) \psi_{n,\varepsilon}(x) \dx  \right )^2 = 1,
	\end{equation}
	equivalently,
	\begin{equation*}
		t_\varepsilon = \frac{1}{\lambda_{n,\varepsilon}} \log \left  ( \left \lvert \int \varrho_{\varepsilon,0}(x) \psi_{n,\varepsilon}(x) \dx  \right \rvert \right ).
	\end{equation*}
	The time window is given by $w_\varepsilon:= \varepsilon^{\frac{1-\gamma}{1+\gamma}}$ and the profile function $\mathfrak{p}(r):=e^{-\lambda_n r}$, $r\in \mathbb{R}$.

	\noindent
	For $x_0 = 0$, let
	\begin{equation*}
		E_n:=\{1\leq j\leq n:\ \psi_j(0)\neq 0\}.
	\end{equation*}
	If $E_n=\emptyset$, then $d_{n,\varepsilon}(t,0)=0$ for all $t\geq 0$. Otherwise, letting $m:=\max E_n$, the centered initial datum exhibits profile cut-off with cut-off time determined by
	\begin{equation*}
		e^{-2\lambda_{m,\varepsilon} t_\varepsilon}\left(\int \varrho_{\varepsilon,0}(x)\psi_{m,\varepsilon}(x)\,\dx\right)^2=1,
	\end{equation*}
	the same time window $w_\varepsilon:=\varepsilon^{\frac{1-\gamma}{1+\gamma}}$, and profile function $\mathfrak p(r)=e^{-\lambda_m r}$.
\end{theorem}

The proof  is provided in \cref{sec:menor}.

\begin{remark}
	\textit{For $x_0\neq 0$} one might expect the cut-off time to be set by the spectral gap $\lambda_1$.  However, for $\gamma\in(0,1/3]$ the \emph{leading} WKB term in $c_{k,\varepsilon}$ (see \cref{lem:pointwise_smallgamma}) contributes 
	\[
		\exp(\varepsilon^{(\gamma-1)/(1+\gamma)}\; \lambda_k\; |x_0|^{1-\gamma}/(1-\gamma)),
	\]
	which blows up at the same rate for all $k$ but at different heights.  It is the \emph{sub-leading} WKB correction (involving $\lambda_k^2$) that separates modes: for any $i<n$, the $i$-th mode's contribution at the $\lambda_n$-defined cut-off time satisfies for some constant $C_\gamma>0$ that
	\[
	  \log T_{i,\varepsilon}
	  \;\leq\; \lambda_i(\lambda_i-\lambda_n)\,C_\gamma\,
	  \varepsilon^{(3\gamma-1)/(1+\gamma)}\,|x_0|^{1-3\gamma} + \mathsf{O}(1) \;\to\; -\infty,
	\]
	since $\lambda_i(\lambda_i-\lambda_n)<0$ for $i<n$ and the exponent $(3\gamma-1)/(1+\gamma)<0$.  
\end{remark}

\begin{remark}[Explicit asymptotic bounds for the cut-off time for $\gamma\in (0,1/3\rbrack$]\hfill

	\noindent
	The mixing time $t_\varepsilon$ in \cref{eq:defThatnuev} can be explicitly estimated via the general mixing-time bounds in the appendix; see \cref{lem:time_smallgamma} and in particular, \cref{eq:teps_bounds} and the exponent analysis following. Specifically, for small enough $\varepsilon>0$ there exists positive constants $C_1,C_2>0$ (depending on $\gamma$ and $n$) such that
	\begin{equation*}
		\frac{-C_1}{\lambda_{n}} \varepsilon^{\frac{1-\gamma}{1+\gamma}}+ \frac{|x_0|^{1-\gamma}}{1-\gamma}
		- C_2 \varepsilon^{\frac{2\gamma}{1+\gamma}}
		\leq
		t_\varepsilon
		\leq
		\frac{C_1}{\lambda_{n} } \varepsilon^{\frac{1-\gamma}{1+\gamma}} + \frac{|x_0|^{1-\gamma}}{1-\gamma}
		+ C_2 \varepsilon^{\frac{2\gamma}{1+\gamma}}.
	\end{equation*}
	Consequently, $t_\varepsilon=\mathsf{O}(1)$, as $\varepsilon\to 0$.
\end{remark}

The third main result of this manuscript is the following, which yields no cut-off phenomenon for $\gamma>1$.
\begin{theorem}
	[No cut-off phenomenon for Galerkin projection, $\gamma>1$]\label{thm:main:2}\hfill

	\noindent
	Assume that $\gamma \in (1,\infty)$, $x_0 \in \R$ and $\delta(\varepsilon) = 1$. Then for any fixed $n\in \mathbb{N}$ there is no cut-off in the sense of \cref{eq:cutdef}.
\end{theorem}

The proof is presented in \cref{sec:proofdos}.

\begin{remark}[Ornstein--Uhlenbeck process]
	\hfill

	\noindent
	We note that $\gamma = 1$ recovers the celebrated Ornstein--Uhlenbeck process, and we can establish a result similar to \cref{thm:main} for $x_0 \neq 0$. In fact, since in this case all the eigenfunctions are explicit (Hermite polynomials) one can  see that there is profile cut-off. In the case $x_0 = 0$ there is no cut-off which can be calculated explicitly from the Hermite polynomials.
\end{remark}

\subsection*{Finding general techniques to ensure the appearance of the cut-off phenomenon}\hfill

\noindent
The cut-off phenomenon has been extensively studied for finite or countable-state Markov chains, both in total variation, in $L^2$-type distances and in separation discrepancy; see, for instance,~\cite{CSC,ChenGYSaloff2024}, the reversible-chain results in~\cite{BHP,ChenHsuSheu}, the separation-cut-off example~\cite{AKL}, and the recent diffusion perspective in~\cite{Salez}. In those works, general criteria are typically formulated in terms of spectral information, geometric input, or suitable mixing inequalities.
The cut-off phenomenon for small-noise perturbations of differential equations with coercive vector field is introduced in the seminal works~\cite{BJ,BJ1}; related continuous-state small-noise behavior, including gradual convergence for degenerate Langevin dynamics, is studied in~\cite{BCJ}. Recently, cutoff phenomenon for spatially localized initial conditions is studied in~\cite{BorMan}.
Our results fit within this general program, but they address a rather different regime: a continuous-state reversible diffusion with a small-noise parameter and a localized initial condition, studied through a truncated $\chi^2$-distance associated with Galerkin projections of the Fokker--Planck eigensystem.

The first novelty of the present work is methodological. We develop a PDE-based spectral approach which combines:
(i) a scaling reduction from $\mathcal L_\varepsilon$ to $\mathcal L_1$;
(ii) a WKB analysis for the associated Sturm--Liouville problem; and
(iii) sharp asymptotics for the Fourier coefficients of localized initial data.
This makes it possible to derive quantitative information not only on the eigenvalues, but also on the pointwise growth of the eigenfunctions at the $\varepsilon$-dependent spatial scale selected by the initial datum. To the best of our knowledge, this level of quantitative control has not been previously used to prove cut-off results for this class of Fokker--Planck dynamics. Moreover, such explicit analysis yields concrete information on the mixing-time scale and, in some regimes, on the limiting profile, going beyond the product condition and the cut-off window alone.

The second novelty is conceptual. In the classical reversible setting, the product condition, namely that the spectral gap multiplied by the mixing time diverges, is a necessary benchmark in many situations, although it is not sufficient in general; see, e.g.,~\cite{BHP,CSC,Levin}. Our results are consistent with that principle in the regimes where cut-off occurs. However, the analysis also shows that, in the present small-noise diffusion problem, the spectral gap alone does not determine the phenomenon if we are interested in the cut-off time. What is decisive is the interplay between eigenvalues, eigenfunction growth, and the symmetry of the initial condition. In particular, the distinction between the regimes $\gamma\in(0,1/3]$ and $\gamma\in(1/3,1)$, as well as cut-off for $x_0=0$, is detected through refined asymptotics of the Fourier coefficients and not by the spectral gap alone.

In summary, our results should therefore be viewed less as a contribution to the general theory of Chen--Saloff-Coste~\cite{ChenSCL,CSC,ChenGYSaloff2024} and related works~\cite{ChenGY,ChenHsuSheu}, and more as a finer asymptotic analysis of a concrete continuous-state model, complementary to explicit continuous-state profile-cut-off examples such as~\cite{BARESQ}. Rather than providing an abstract criterion for an arbitrary family of chains or diffusions, we identify a concrete continuous-state model for which the cut-off time, the window, and, in the regime $\gamma\in(0,1/3]$, the limiting profile can be computed from detailed semiclassical information on the eigensystem. In this sense, the paper shows that refined PDE and spectral asymptotic methods can provide a viable route to cut-off results beyond the finite-state framework.

Finally, although we restrict attention here to one-dimensional monomial convex potentials, this class already exhibits several non-trivial features: singular small-noise scaling, unbounded $\chi^2$-distance, sensitivity to the parity of the eigenfunctions, and a genuine change of behavior across parameter regimes (presence of resonances and a phase transition). For this reason, the present work may be regarded as a first step toward a broader cut-off theory for reversible diffusions based on quantitative spectral asymptotics.

The rest of the manuscript is organized as follows: In \cref{Sec:outline}, we present the main tools to prove the main theorems: \cref{thm:main} and \cref{thm:main:2}. In \cref{sec:auxiliary}, we compute the asymptotic growth of the Fourier coefficients and asymptotics of the mixing times for $\frac{1}{3}  < \gamma < 1$. In \cref{sec:growtheigen}, we calculate the asymptotic growth of the eigenfunctions. Finally, we provide an appendix with auxiliary results and proofs that have been used throughout the manuscript. The appendix is divided into three parts. More specifically, in \cref{sec:srdist} we give the spectral representation of the $\chi^2$--distance to equilibrium, in \cref{sec:menor} we provide the proof for the case $0<\gamma\leq \frac{1}{3}$, and lastly in \cref{ap:tools} we show the scaling argument used in \cref{lem:g:scaling} below.

\section{{Outline of the proofs}}\label{Sec:outline}
In this section, we present the main ideas for the proofs. It is divided in 4 subsections as follows.
In \cref{sec:sarg} we present a scaling argument, which allows us to reduce the problem to the study of the spectrum of $\mathcal{L}_1$, that is, $\varepsilon=1$. In \cref{sec:geig} we apply the so-called WKB expansion to study the growth of the eigenfunctions of $\mathcal{L}_1$. In \cref{sec:auxr} we present the growth of the Fourier coefficients. Finally, in \cref{subsection:proofs} we present the proofs of the main results: \cref{thm:main,thm:main:2}.

\subsection{{Scaling argument}}\label{sec:sarg}
\hfill

\noindent
In this subsection, we
point out a scaling property of the generator $\mathcal{L}_\varepsilon$ defined in \cref{eq:FP} inherited straightforwardly from the scaling of the potential $V$ defined in \cref{def:V}.
\begin{lemma}[Scaling of the spectrum]\label{lem:g:scaling}
	\hfill

	\noindent
	Let $u:\R \to \R$ be a solution to $\mathcal{L}_1 u = f$ for a suitable smooth function $f$. Then the function defined by $u_\varepsilon(x) = u(x \varepsilon^{-\frac{1}{1+\gamma}})$ satisfies
	\begin{equation*}
		\mathcal{L}_\varepsilon u_\varepsilon = \varepsilon^{\frac{\gamma-1}{\gamma+1}} f\left ( x \varepsilon^{-\frac{1}{1+\gamma}} \right ).
	\end{equation*}
	In particular,  any eigenfunction $u$ of $-\mathcal{L}_1 u = \lambda u$ with an eigenvalue $\lambda$ can be transformed into the eigenfunction $u_\varepsilon$ of $-\mathcal{L}_\varepsilon$ with eigenvalue $\lambda \varepsilon^{\frac{\gamma-1}{\gamma+1}}$.

	Furthermore, the following relation holds true $\|u\|_{L^2(C_1 e^{-V})} = \|u_\varepsilon\|_{L^2(C_\varepsilon e^{-V/\varepsilon})}$.
\end{lemma}
The proof is given in \cref{ap:tools}.

The above scaling has the special case $\gamma=1$ (corresponding to the Ornstein--Uhlenbeck process) and is the critical case when the eigenvalues of $\mathcal{L}_\varepsilon$ no longer depends on $\varepsilon$.

Looking at \cref{eq:initialcond,eq:dnvarepsilon_intro} we see that in order to study $d_{n,\varepsilon}$ we need to know the blow up rate of the eigenfunctions $u_\varepsilon$ around $x_0$ as $\varepsilon \to 0$. By \cref{lem:g:scaling} this is equivalent to studying the asymptotic growth of $u$ as $x \to \infty$, which is the subject of the next subsection.

\subsection{{The growth of the eigenfunctions}}\label{sec:geig}
\hfill

\noindent
In this subsection, we establish the asymptotic growth of the eigenfunctions for the operator $-\mathcal{L}_1$, that is, for $\varepsilon=1$. For the special case $\gamma=\tfrac12$, related asymptotics are stated in~\cite{MHR}.
Here we give a rigorous derivation for general $\gamma>0$ in the form needed for the cut-off analysis.

We start by recalling the classical transformation of the Sturm--Liouville operator to a Schr\"odinger equation. More precisely,  the Sturm--Liouville equation $-\mathcal{L}_1u=\lambda u$, i.e.,
\begin{equation}\label{eq:SL}
	-u''(x) + V'(x) u'(x) = \lambda u(x)
\end{equation}
satisfied in a classical sense at $x$, for $V$ twice continuously differentiable at $x$, can be transformed into the Schrödinger equation
\begin{equation} \label{eq:Wittenprime}
	v''(x) = \left ( \frac{|V'(x)|^2}{4} - \frac{1}{2} V''(x) - \lambda \right ) v(x),
\end{equation}
where $v(x) = e^{-\frac{1}{2} V(x)} u(x)$. Furthermore, if $u \in L^2(e^{-V(x)}\dx)$ then $v \in L^2(\dx)$.

The following lemma allows us to write the solution of $v$ in \cref{eq:Wittenprime} using the so-called Wentzel--Kramers--Brillouin expansion (WKB  expansion).
\begin{lemma}[WKB expansion]\label{lem:WKBprime}
	\hfill

	\noindent
	Let $\lambda>0$ be fixed and $n \in \N$. Assuming \cref{eq:spectralcondition} there exists a turning point $x^\ast$ of
	\begin{equation*}
		T(x):=\frac{|V'(x)|^2}{4} - \frac{V''(x)}{2} - \lambda,
	\end{equation*}
	that is,  the point $x^*$ for which $T(x)\geq 0$ for all $x \geq x^*$.
	Denote $f = V'$ and let $f \in \mathcal{C}^1([x^*,\infty))$.
	Let $v$ be a solution to \cref{eq:Wittenprime} in $[x^*,\infty)$. Write $v(x) = e^{\sum_{i=0}^{n} S_i(x)+R_n(x)}$, where
	\begin{equation}
		\begin{split}\label{ec:S0S1Sk}
			S_0(x)=   & -\int_{x^*}^{x} \frac{f(s)}{2}
			\ds, \quad
			S_1(x)=  \int_{x^*}^x \frac{\lambda}{f(s)} \ds,                                                                                           \\
			S_k'(x) = & \frac{1}{f(x)} S_{k-1}''(x) + \frac{1}{f(x)} \sum_{j=1}^{k-1} S_j'(x) S_{k-j}'(x), \quad S_k(x^*) = 0, \quad 2 \leq k \leq n,
		\end{split}
	\end{equation}
	for all $x\geq x^*$.
	Then the remainder term $R_{n}$ satisfies for $\widehat u =e^{R_{n}}$
	\begin{equation} \label{eq:SLL}
		-\widehat u''(x) - 2\widehat u'(x) \sum_{j=0}^{n} S_j'(x) = \widehat u(x)  \left ( S_n''(x) +\sum_{n_1+n_2 \geq n+1}^{n_1 \leq n; n_2 \leq n} S_{n_1}'(x)S_{n_2}'(x) \right ),
	\end{equation}
	with initial condition $\widehat u(x^*) = v(x^*)$.
\end{lemma}
The proof is given in \cref{sec:WKBproof}.

Along this manuscript, we use the following  notation.
We say that $f\lesssim g$ on a domain $D$ if there exists a positive (implied) constant  $C$ such that $f(x)\leq Cg(x)$ for all $x\in D$.
We also say that $f\approx g$ on a domain $D$ if there exists a (implied) constant $C\geq 1$ such that
\begin{equation*}
	\frac{1}{C}g(x)\leq f(x)\leq Cg(x)\quad \textrm{ for all }\quad x\in D.
\end{equation*}
The following theorem allows us to find up to a constant, the asymptotic behavior of $u$ in \cref{eq:SL}.
\begin{theorem}[Growth of the eigenfunctions, $\varepsilon=1$ and $\gamma>0$]\label{thm:main:3prime}
	\hfill

	\noindent
	Let $V$ be the potential defined in \cref{def:V} for some $\gamma>0$ and let $\varrho_{1,\infty}$ be the density function given in \cref{ec:invm} for $\varepsilon=1$. Let $u$ be a solution to the eigenvalue problem
	\begin{equation*}
		-u'' + V' u' = \lambda u \quad \text{ on }\quad  \R
	\end{equation*}
	for some eigenvalue $\lambda > 0$ and $u \in L^2_{\varrho_{1,\infty}}(\R)$. Let $x^*$ be the turning point defined in \cref{lem:WKBprime}. Then if $n > \frac{1+\gamma}{2\gamma}$ and $u(x^*) \geq 0$ we then have  on $[x^*,\infty)$
	\begin{equation*}
		u(x) \approx e^{\sum_{i=1}^n S_i(x)},
	\end{equation*}
	where $S_i(x)$ are the solutions to the closed hierarchy given in \cref{lem:WKBprime}.
	Moreover, there exist constants $C_2,C_3,\ldots,C_n \in \R$ such that
	\begin{equation*}
		u(x)\approx
		\begin{cases}
			e^{\lambda \frac{|x|^{1-\gamma}}{1-\gamma} + \sum_{i=2}^{n} C_i |x|^{1-(2i-1)\gamma}},                     & \textrm{ if } \gamma^{-1} \notin 2\mathbb{N}-1,
			\\
			e^{\lambda \frac{|x|^{1-\gamma}}{1-\gamma} + \sum_{i=2}^{m-1} C_i |x|^{1-(2i-1)\gamma} + C_{m} \log(|x|)}, & \textrm{if } \gamma = \frac{1}{2m-1} \textrm{ for some } m\in\mathbb{N},\ 2 \leq m < n.
		\end{cases}
	\end{equation*}
\end{theorem}
The proof is given in \cref{sec:growth}.

\begin{remark}[Error term and the appearing of resonances]\hfill

	\noindent
	When $\gamma > \nicefrac{1}{3}$ the expression above becomes particularly simple since we can take $n = 2$. In this case
	$S_1(x) = \lambda \frac{|x|^{1-\gamma}}{1-\gamma} + C_1$, $S_2(x) = \frac{\lambda}{2}\,|x|^{-2\gamma} + \frac{\lambda^2}{1-3\gamma} |x|^{\,1-3\gamma}+ C$. The point is that $S_2$ will not contribute more than a multiplicative constant in $u$ as it is bounded. As we will prove in \cref{sec:growtheigen} the remainder term $R_2$ is also bounded.

	In the above theorem we also have a phase transition at every resonant point of $\gamma$, the first resonant point is at $\gamma = \nicefrac{1}{3}$ (we have one at $\gamma = \nicefrac{1}{k}$ for all odd $k$). In this case we need to take $n=3$ so that the remainder is bounded, while the logarithmic term already appears in $S_2$. More precisely,
	$S_2(x) = \frac{\lambda}{2} |x|^{-2/3} + \lambda^2 \log |x| + C$, where the first term only contributes a multiplicative constant as it is bounded on $[x^\ast,\infty)$. $S_3$ can be calculated explicitly and is bounded, and the remainder $R_3$ is also bounded.
\end{remark}

As a corollary we obtain the following asymptotics when $\frac{1}{3}<\gamma<1$.
\begin{corollary}[Growth of the eigenfunctions, $\varepsilon>0$ and $\frac{1}{3}<\gamma<1$]\label{cor:asymptoticsprime}
	\hfill

	\noindent
	Let $V$ be the potential defined in \cref{def:V} for some $\gamma>0$ and let $\varrho_{\varepsilon,\infty}$ be the density function given in \cref{ec:invm}.
	Assume that $\frac{1}3 < \gamma < 1$ and let $v \in L^2_{\varrho_{\varepsilon,\infty}}$ be a solution to the eigenvalue problem
	\begin{equation*}
		-\varepsilon v'' + V' v' = \lambda_{\varepsilon} v \quad \text{on } \R,
	\end{equation*}
	for $\lambda_\varepsilon = \lambda \varepsilon^{\frac{\gamma-1}{\gamma+1}} > 0$ with $\lambda$ being an eigenvalue of $\mathcal{L}_1$.
	Then it follows that
	\begin{equation*}
		v(x)
		\approx e^{\lambda \frac{|x \varepsilon^{-\frac{1}{1+\gamma}}|^{1-\gamma}}{1-\gamma}}\quad \textrm{ in }\quad
		[x^* \varepsilon^{\frac{1}{1+\gamma}},\infty),
	\end{equation*}
	where $x^*$ is the turning point defined in \cref{thm:main:3prime}.
\end{corollary}
The proof is presented in \cref{sec:growth}.

\subsection{Growth of the Fourier coefficients and mixing times asymptotics for \texorpdfstring{$\frac{1}{3} < \gamma < 1$}{1/3 < gamma < 1}}\label{sec:auxr}
\hfill

\noindent
For simplicity of the presentation, with the help of \cref{cor:asymptoticsprime}
we study the growth behavior of the Fourier coefficients
\begin{equation*}
	\int \varrho_{\varepsilon,0}(x) \psi_{k,\varepsilon}(x) \dx,
\end{equation*}
where $\varrho_{\varepsilon,0}$ is the initial density given in \cref{eq:initialcond}
when $\frac{1}{3}<\gamma<1$.
We also provide asymptotics for the mixing times when $\frac{1}{3}<\gamma<1$.
The case $0<\gamma\leq \frac{1}{3}$ is given in \cref{sec:menor}.

\begin{lemma}[Growth of the Fourier coefficients $\frac{1}{3}<\gamma<1$ and $\varepsilon>0$]\label{lem:pointwise_intro}
	\hfill

	\noindent
	Let $\frac{1}{3}  < \gamma < 1$ and let $\delta = \varepsilon^{\frac{1-\gamma}{\gamma+1}}$, and let $\psi_{k,\varepsilon}$ be the $k$-th non-constant eigenfunction of the operator $-\mathcal{L}_\varepsilon$ with eigenvalue $\lambda_{k,\varepsilon} = \lambda_k \varepsilon^{\frac{\gamma-1}{\gamma+1}}$, where $\lambda_k$ is the $k$-th eigenvalue of the operator $-\mathcal{L}_1$.
	For $x_0\neq 0$
	and small enough $\varepsilon = \varepsilon(k,\gamma,x_0)>0$ it holds
	\begin{equation}\label{eq:xnozero}
		\left\lvert \int \varrho_{\varepsilon,0}(x) \psi_{k,\varepsilon}(x) \dx \right\rvert \approx \fint_{B_{\delta}(x_0)} e^{\lambda_{k} \frac{|x \varepsilon^{-\frac{1}{\gamma+1}}|^{1-\gamma}}{1-\gamma}} \dx \approx e^{\lambda_{k} \varepsilon^{\frac{\gamma-1}{\gamma+1}} \frac{|x_0|^{1-\gamma}}{1-\gamma}},
	\end{equation}
	where the implied constant is independent of $\varepsilon$.

	For  $x_0 = 0$ and $\psi_{k}(0) \neq 0$ we have that $\psi_k$ is an even function and hence
	\begin{equation*}
		\left\lvert \int \varrho_{\varepsilon,0}(x) \psi_{k,\varepsilon}(x) \dx \right\lvert  \approx \varepsilon^{\frac{\gamma(1-\gamma)}{1+\gamma}} e^{\frac{\lambda_{k}}{1-\gamma} \varepsilon^{\frac{\gamma(\gamma-1)}{\gamma+1}} },
	\end{equation*}
	while  for $x_0 = 0$ and  $\psi_{k}(0) = 0$ we have that $\psi_k$ is an odd function and hence
	\begin{equation*}
		\int \varrho_{\varepsilon,0}(x) \psi_{k,\varepsilon}(x) \dx = 0.
	\end{equation*}
\end{lemma}
The proof is given in \cref{sec:auxiliary}.

Now we can define the unique time $t_\varepsilon$ such that
\begin{equation} \label{eq:mixing_0}
	d_{1,\varepsilon}(t_\varepsilon,x_0) = 1,
\end{equation}
where $d_{1,\varepsilon}$ is defined in \cref{eq:Ga}.

\begin{lemma}[Mixing times asymptotics for $\frac{1}{3}<\gamma<1$]\label{lem:time_intro}
	\hfill

	\noindent
	Let $\frac{1}{3}  < \gamma < 1$, $x_0 \neq 0$ and let $t_\varepsilon$ be the time defined by \cref{eq:mixing_0}. Given $1 \leq n < \infty$, then for small enough $\varepsilon$ it holds
	\begin{equation*}
		d_{n,\varepsilon}(t_\varepsilon,x_0) \approx 1
	\end{equation*}
	with implied constant independent of $\varepsilon$.
	Furthermore, for small enough $\varepsilon$ it holds for a constant $C:=C(\gamma)\geq 1$
	\begin{equation} \label{eq:mixingtime:bound}
		\frac{\log(1/C)}{\lambda_{1}} \varepsilon^{\frac{1-\gamma}{1+\gamma}}+ \frac{|x_0|^{1-\gamma}}{1-\gamma}
		\leq
		t_\varepsilon
		\leq
		\frac{\log(C)}{\lambda_{1} } \varepsilon^{\frac{1-\gamma}{1+\gamma}} + \frac{|x_0|^{1-\gamma}}{1-\gamma}.
	\end{equation}

	If, on the other hand, $x_0 = 0$, define
	\begin{equation*}
		E_n:=\{1\leq j\leq n:\ \psi_j(0)\neq 0\}.
	\end{equation*}
	If $E_n=\emptyset$, then $d_{n,\varepsilon}(t,0)=0$ for all $t\geq 0$. Otherwise, let $m:=\max E_n$ and define $t_\varepsilon$ by the condition
	\begin{equation*}
		e^{-2\lambda_m \varepsilon^{\frac{\gamma-1}{1+\gamma}} t_\varepsilon}
		\left(\int \varrho_{\varepsilon,0}(x)\psi_{m,\varepsilon}(x)\,\dx\right)^2 = 1.
	\end{equation*}
	Then, for $\varepsilon$ sufficiently small, we have
	\begin{equation*}
		d_{n,\varepsilon}(t_\varepsilon,0) \approx 1.
	\end{equation*}
	In addition, for small enough $\varepsilon$ it holds
	\begin{multline} \label{eq:mixingtime:bound:zero}
		\frac{\varepsilon^{\frac{1-\gamma}{1+\gamma}}}{\lambda_m}
		\left (\frac{\gamma(1-\gamma)}{1+\gamma} \log(\varepsilon)-\log(C) \right ) + \frac{1}{1-\gamma}\varepsilon^{\frac{(\gamma-1)^2}{\gamma+1}}
		\leq
		t_\varepsilon
		\\
		\leq
		\frac{\varepsilon^{\frac{1-\gamma}{1+\gamma}}}{\lambda_m}
		\left (\frac{\gamma(1-\gamma)}{1+\gamma} \log(\varepsilon)+\log(C) \right )
		+ \frac{1}{1-\gamma}\varepsilon^{\frac{(\gamma-1)^2}{\gamma+1}}.
	\end{multline}
\end{lemma}
The proof is given in \cref{sec:auxiliary}.

\subsection[Proof of main results]{{Proof of the main results: \cref{thm:main} and \cref{thm:main:2}}}\label{subsection:proofs}
\hfill

\noindent
In this subsection, using the results stated in \cref{sec:sarg,sec:geig,sec:auxr} we give the proofs of \cref{thm:main,thm:main:2}.

\subsubsection{{Proof of \cref{thm:main}}}\label{sec:proofuno}\hfill

\noindent
Assume that $x_0\neq 0$.
Define the vectors
\begin{align*}
	V_{n,\varepsilon} & := \left \{ e^{-2\lambda_k \varepsilon^{\frac{\gamma-1}{1+\gamma}} t_\varepsilon} \left ( \int \varrho_{\varepsilon,0}(x)\psi_{k,\varepsilon}(x) \dx \right )^2 \right \}_{k=1}^n,
	\\
	W_n               & := \left \{ e^{-2\lambda_k \varepsilon^{\frac{\gamma-1}{1+\gamma}} w_\varepsilon r} \right \}_{k=1}^n = \left \{ e^{-2\lambda_k r} \right \}_{k=1}^n,
\end{align*}
where the window-size $w_\varepsilon = \varepsilon^{\frac{1-\gamma}{1+\gamma}}$ and $r \in \R$ is a parameter. Then we have
\begin{equation*}
	d_{n,\varepsilon}^2 (t_\varepsilon + r w_\varepsilon,x_0) =  V_{n,\varepsilon} \cdot W_n,
\end{equation*}
where $\cdot$ denotes the usual inner product in $\mathbb{R}^n$.
Applying \cref{lem:pointwise_intro,lem:time_intro} we see that each entry in $V_{n,\varepsilon}$ satisfies for small enough $\varepsilon > 0$, there is a constant $C:=C(\gamma) \geq 1$ such that
\begin{equation*}
	\frac{1}{C} e^{-2\frac{\lambda_k \log(C)}{\lambda_1}} \leq e^{-2\lambda_k \varepsilon^{\frac{\gamma-1}{1+\gamma}} t_\varepsilon} \left ( \int \varrho_{\varepsilon,0}(x)\psi_{k,\varepsilon}(x) \dx \right)^2 \leq
	C e^{2\frac{\lambda_k \log(C)}{\lambda_1}}.
\end{equation*}
Thus,
\begin{align*}
	\limsup_{\varepsilon \to 0} d_{n,\varepsilon} (t_\varepsilon + r w_\varepsilon ,x_0) & \leq \sqrt{(\limsup_{\varepsilon \to 0} V_{n,\varepsilon}) \cdot W_n} =: g_+(r)<\infty,
	\\
	\liminf_{\varepsilon \to 0} d_{n,\varepsilon} (t_\varepsilon + r w_\varepsilon ,x_0) & \geq \sqrt{(\liminf_{\varepsilon \to 0} V_{n,\varepsilon}) \cdot W_n} =: g_-(r)>0.
\end{align*}

We now write out the separate centered case $x_0=0$. Let
\begin{equation*}
	E_n:=\{1\leq j\leq n:\ \psi_j(0)\neq 0\}.
\end{equation*}
If $E_n=\emptyset$, then by parity every Fourier coefficient in \cref{eq:dnvarepsilon_intro} vanishes and hence $d_{n,\varepsilon}(t,0)=0$ for all $t\geq 0$. Assume now that $E_n\neq\emptyset$ and let
\begin{equation*}
	m:=\max E_n.
\end{equation*}
Then $\psi_m$ is even, and every odd mode has zero coefficient. Define $t_\varepsilon$ by
\begin{equation*}
	e^{-2\lambda_m \varepsilon^{\frac{\gamma-1}{1+\gamma}} t_\varepsilon}
	\left(\int \varrho_{\varepsilon,0}(x)\psi_{m,\varepsilon}(x)\,\dx\right)^2=1,
\end{equation*}
and let $w_\varepsilon:=\varepsilon^{\frac{1-\gamma}{1+\gamma}}$. By the centered part of \cref{lem:pointwise_intro}, for every $j\in E_n$,
\begin{equation*}
	\left\lvert \int \varrho_{\varepsilon,0}(x)\psi_{j,\varepsilon}(x)\,\dx\right\rvert
	\approx
	\varepsilon^{\frac{\gamma(1-\gamma)}{1+\gamma}}
	e^{\frac{\lambda_j}{1-\gamma}\varepsilon^{\frac{\gamma(\gamma-1)}{\gamma+1}}}.
\end{equation*}
Hence, for every $j\in E_n$ with $j<m$,
\begin{align*}
	 & -2\lambda_j \varepsilon^{\frac{\gamma-1}{1+\gamma}} t_\varepsilon
	+2\frac{\lambda_j}{1-\gamma}\varepsilon^{\frac{\gamma(\gamma-1)}{\gamma+1}}
	+2\frac{\gamma(1-\gamma)}{1+\gamma}\log\varepsilon
	\\
	 & \qquad=
	2\frac{\gamma(1-\gamma)}{1+\gamma}\Bigl(1-\frac{\lambda_j}{\lambda_m}\Bigr)\log\varepsilon+\mathsf O(1),
\end{align*}
which tends to $-\infty$ as $\varepsilon\to 0$ because $\lambda_j<\lambda_m$ and $\log\varepsilon\to -\infty$. Therefore,
\begin{equation*}
	e^{-2\lambda_j \varepsilon^{\frac{\gamma-1}{1+\gamma}} t_\varepsilon}
	\left(\int \varrho_{\varepsilon,0}(x)\psi_{j,\varepsilon}(x)\,\dx\right)^2\to 0
	\qquad \text{for every } j<m.
\end{equation*}
Since the odd coefficients vanish exactly, the $m$-th mode is the unique non-vanishing contribution in the limit. Consequently, for any fixed $r\in\mathbb R$,
\begin{equation*}
	d_{n,\varepsilon}^2(t_\varepsilon+r w_\varepsilon,0)
	=
	\sum_{j=1}^n e^{-2\lambda_j \varepsilon^{\frac{\gamma-1}{1+\gamma}}r w_\varepsilon}
	e^{-2\lambda_j \varepsilon^{\frac{\gamma-1}{1+\gamma}} t_\varepsilon}
	\left(\int \varrho_{\varepsilon,0}(x)\psi_{j,\varepsilon}(x)\,\dx\right)^2
	\to e^{-2\lambda_m r},
\end{equation*}
and hence
\begin{equation*}
	d_{n,\varepsilon}(t_\varepsilon+r w_\varepsilon,0)\to e^{-\lambda_m r}.
\end{equation*}
Thus, the centered case is governed by the highest even mode present in the truncation and exhibits window cut-off (indeed profile cut-off) with profile $e^{-\lambda_m r}$ whenever such a mode exists.

\subsubsection{{Proof of \cref{thm:main:2}}}\label{sec:proofdos}\hfill

\noindent
Fix $1\leq k\leq n$. By \cref{lem:g:scaling}, the $k$-th eigenfunction of $\mathcal{L}_\varepsilon$ has the form
\begin{equation*}
	\psi_{k,\varepsilon}(x)=\psi_k\bigl(x\varepsilon^{-\frac{1}{1+\gamma}}\bigr),
\end{equation*}
where $\psi_k$ is the corresponding eigenfunction of $\mathcal{L}_1$.
Since $\gamma>1$, every exponent $1-(2i-1)\gamma$ appearing in the expansion of \cref{thm:main:3prime} is negative. Hence, \cref{thm:main:3prime} shows that $\psi_k$ is bounded on $[x_k^\ast,\infty)$, where $x_k^\ast$ is its turning point. Applying the same argument to $x\mapsto \psi_k(-x)$ yields boundedness on $(-\infty,-x_k^\ast]$, and continuity gives boundedness on the remaining compact interval. Therefore,
\begin{equation*}
	\|\psi_k\|_{L^\infty(\mathbb{R})}<\infty.
\end{equation*}
By the scaling relation above,
\begin{equation*}
	\|\psi_{k,\varepsilon}\|_{L^\infty(\mathbb{R})}=\|\psi_k\|_{L^\infty(\mathbb{R})}
\end{equation*}
for every $\varepsilon>0$. Define
\begin{equation*}
	M_n:=\max_{1\leq k\leq n}\|\psi_k\|_{L^\infty(\mathbb{R})}<\infty.
\end{equation*}
Since $\delta(\varepsilon)=1$, the initial datum $\varrho_{\varepsilon,0}$ is a probability density on $B_1(x_0)$, and therefore
\begin{equation*}
	\left|\int_{\mathbb{R}} \varrho_{\varepsilon,0}(x)\psi_{k,\varepsilon}(x)\,\dx\right|
	\leq \|\psi_{k,\varepsilon}\|_{L^\infty(\mathbb{R})}
	\leq M_n.
\end{equation*}
Using the spectral representation from \cref{cor:separation}, for every $t\geq 0$ we obtain
\begin{equation*}
	d_{n,\varepsilon}(t,x_0)^2
	=
	\sum_{k=1}^n e^{-2\lambda_{k,\varepsilon}t}
	\left(\int_{\mathbb{R}} \varrho_{\varepsilon,0}(x)\psi_{k,\varepsilon}(x)\,\dx\right)^2
	\leq
	\sum_{k=1}^n M_n^2
	=
	nM_n^2.
\end{equation*}
Hence,
\begin{equation*}
	d_{n,\varepsilon}(t,x_0)\leq \sqrt{n}\,M_n
\end{equation*}
uniformly in $t\geq 0$, $\varepsilon>0$, and $x_0\in\mathbb{R}$.
If cut-off in the sense of \cref{eq:cutdef} were to hold for some timescale $t_\varepsilon$, then for each $c\in(0,1)$ we would need
\begin{equation*}
	d_{n,\varepsilon}(ct_\varepsilon,x_0)\to \infty,
\end{equation*}
as $\varepsilon\to 0$. This is impossible because of the uniform bound above. Therefore, there is no cut-off phenomenon for $\gamma>1$.

\section{The growth of the Fourier coefficients and asymptotics of the mixing times for \texorpdfstring{$\frac{1}{3}  < \gamma < 1$}{1/3 < gamma < 1}}\label{sec:auxiliary}

In this section, we provide the proofs of \cref{lem:pointwise_intro} and \cref{lem:time_intro}.

\subsection{Growth of the Fourier coefficients \texorpdfstring{$\frac{1}{3}<\gamma<1$ and $\varepsilon>0$}{1/3 < gamma < 1 and epsilon > 0}}
\begin{proof}[Proof of \cref{lem:pointwise_intro}]
	Along the proof, let $x^*_k$ be the turning-point of the $k$-th eigenfunction $\psi_k$ as defined in~\cref{lem:WKBprime}. Then  \cref{cor:asymptoticsprime} implies for $x \in [x^*_k \varepsilon^{\frac{1}{1+\gamma}},\infty)$ that
	\begin{equation*}
		\psi_{k,\varepsilon}(x) \approx C e^{\lambda_{k} \frac{|x \varepsilon^{-\frac{1}{\gamma+1}}|^{1-\gamma}}{1-\gamma}},
	\end{equation*}
	where the implied constant is independent of $\varepsilon$.
	Since $\delta = \varepsilon^{\frac{1-\gamma}{\gamma+1}}$, for all sufficiently small $\varepsilon$ we have
	\[
		B_{\delta}(x_0) \subset \bigl[\,x_k^*\,\varepsilon^{\frac{1}{\gamma+1}},\,\infty\bigr),
	\]
	so the growth estimate for the eigenfunction applies throughout the ball. By the monotonicity of
	$r \mapsto r^{1-\gamma}$ and the smallness of $\delta$ relative to $|x_0|$, the integrand varies by at most a
	parameter-independent factor on $B_\delta(x_0)$. Consequently, the average over the ball is uniformly comparable
	to the pointwise value at the center:
	\[
		\fint_{B_{\delta}(x_0)}
		\exp \left (\lambda_k\,\tfrac{\bigl|x\,\varepsilon^{-\frac{1}{\gamma+1}}\bigr|^{\,1-\gamma}}{1-\gamma}\right )\,dx
		\approx
		\exp \Bigl(\lambda_k\,\varepsilon^{\frac{\gamma-1}{\gamma+1}}\,\frac{|x_0|^{\,1-\gamma}}{1-\gamma}\Bigr),
	\]
	where the implied constants are independent of $\varepsilon$. Combining this with the preceding estimates yields
	\cref{eq:xnozero}.

	In the sequel, we consider the case $x_0 = 0$.
	We do a change of variables and split the integral into two parts for $\widehat \delta:= \delta \varepsilon^{-1/(1+\gamma)}$ as follows:
	\begin{align*}
		\fint_{B_\delta(0)} \psi_{k,\varepsilon}(x) \dx =
		\frac{1}{\widehat \delta} \int_{-x^*_k}^{x^*_k} \psi_{k}(y) \dy + \frac{1}{\widehat \delta} \int_{B_{\widehat \delta}(0) \setminus (-x^*_k,x^*_k)} \psi_{k}(y) \dy.
	\end{align*}
	The first term tends to $0$ as $\varepsilon \to 0$ since $\widehat \delta^{-1} = \varepsilon^{\gamma/(1+\gamma)} \to 0$. For the second term, an elementary asymptotic estimate of the tail integral yields
	\begin{equation*}
		\lim_{\varepsilon \to 0} \frac{\widehat \delta^{-1} \int_{B_{\widehat \delta}(0) \setminus (-x^*_k,x^*_k)} e^{\lambda_{k} \frac{|y|^{1-\gamma}}{1-\gamma}} \dy}{\widehat \delta^{\gamma-1} e^{\lambda_{k} \frac{\widehat \delta^{1-\gamma}}{1-\gamma}}} = C_0,
	\end{equation*}
	as such for small enough $\varepsilon$ we have that one side of the second term satisfies (up to a sign change)
	\begin{align*}
		\frac{1}{\widehat \delta} \int_{x^*_k}^{\widehat \delta} \psi_{k}(y) \dy \approx \varepsilon^{\frac{\gamma(1-\gamma)}{1+\gamma}} e^{\frac{\lambda_{k}}{1-\gamma} \varepsilon^{\frac{\gamma(\gamma-1)}{\gamma+1}} }.
	\end{align*}

	In order to calculate the full integral we need to know the parity of the eigenfunction.
	Since the potential is even, the reflection $x\mapsto -x$ preserves the eigenspace associated with $\lambda_k$. Because the eigenvalue is simple in one dimension, the reflected eigenfunction must equal either $u$ or $-u$. Thus, the left tail agrees with the right tail up to a sign, and it remains to determine whether $u$ is symmetric or antisymmetric.

	Since the equation is
	\begin{equation*}
		-u'' + V' u' = \lambda_k u
	\end{equation*}
	and $\lambda_k > 0$, then by the maximum principle we have that if $u \geq 0$ on some interval $(z_1,z_2)$ the minimum is attained at the end-points $z_1,z_2$. Assuming that $u(z_1) = u(z_2) = 0$ then, since $V'$ is continuous everywhere, the Hopf Lemma (\cite{Hopf}) tells us that $u'(z_1) > 0$ and $u'(z_2) < 0$.

	By the transformation to the Schr\"odinger operator in \cref{eq:Wittenprime} we know that on each connected component outside the interval $[-x^*_k,x^*_k]$ the eigenfunction $u$ does not have any zeros. From the regularity of the solution and the above argument we can now conclude that the eigenfunction $u$ has a finite number of zeros in $[-x^*_k,x^*_k]$, which are symmetric around $0$. On the one hand, if $u(0) = 0$ then the eigenfunction has an odd number of zeros and thus has to be antisymmetric. On the other hand, if $u(0) \neq 0$ then the eigenfunction has an even number of zeros and thus has to be symmetric.
\end{proof}

\subsection{Mixing times asymptotics for \texorpdfstring{$\frac{1}{3}<\gamma<1$}{1/3 < gamma < 1}}

\begin{proof}[Proof of \cref{lem:time_intro}]
	We start with  the case when $x_0 \neq 0$. By \cref{cor:separation}, $t_\varepsilon$ solves
	\begin{equation*}
		1 = d_{1,\varepsilon}(t_\varepsilon) = e^{-\lambda_{1,\varepsilon} t_\varepsilon} \left \lvert \int \varrho_{\varepsilon,0}(x) \psi_{1,\varepsilon}(x) \dx \right \rvert,
	\end{equation*}
	which by \cref{cor:asymptoticsprime,lem:pointwise_intro} implies that for that $t_\varepsilon$ it holds for small enough $\varepsilon$
	\begin{equation*}
		\frac{1}{C} \leq e^{-\lambda_{1} \varepsilon^{\frac{\gamma-1}{1+\gamma}} t_\varepsilon}
		e^{ \lambda_1 \varepsilon^{\frac{\gamma-1}{1+\gamma}} \frac{|x_0|^{1-\gamma}}{1-\gamma}} \leq C.
	\end{equation*}
	Taking the log and reshuffling we get the bounds in \cref{eq:mixingtime:bound}.

	By using the definition of $d_{n,\varepsilon}(t)$, $t_\varepsilon$ and the asymptotics of $\psi_{k,\varepsilon}$ from \cref{cor:separation,cor:asymptoticsprime} we get for small enough $\varepsilon$
	\begin{equation*}
		(d_{n,\varepsilon}(t_\varepsilon))^2
		=
		\sum_{k=1}^n e^{-2\lambda_k \varepsilon^{\frac{\gamma-1}{1+\gamma}} t_\varepsilon} \left ( \int \varrho_{\varepsilon,0}(x) \psi_{k,\varepsilon}(x) \dx \right )^2
		\approx
		\sum_{k=1}^n e^{-2\lambda_k \varepsilon^{\frac{\gamma-1}{1+\gamma}} t_\varepsilon} e^{2 \lambda_k \varepsilon^{\frac{\gamma-1}{1+\gamma}} \frac{|x_0|^{1-\gamma}}{1-\gamma}}
		\approx 1.
	\end{equation*}

	Now, we treat the case $x_0= 0$. If $E_n=\emptyset$, then every Fourier coefficient vanishes by parity and there is nothing to prove. Assume therefore that $E_n\neq\emptyset$ and let $m:=\max E_n$. As above, the time $t_\varepsilon$ satisfies
	\begin{equation*}
		1 = d_{m,\varepsilon}(t_\varepsilon,0)
		=
		e^{-\lambda_m \varepsilon^{\frac{\gamma-1}{1+\gamma}} t_\varepsilon} \left \lvert \int \varrho_{\varepsilon,0}(x) \psi_{m,\varepsilon}(x) \dx \right \rvert
		\approx
		e^{-\lambda_m \varepsilon^{\frac{\gamma-1}{1+\gamma}} t_\varepsilon} \varepsilon^{\frac{\gamma(1-\gamma)}{1+\gamma}} e^{\frac{\lambda_m}{1-\gamma} \varepsilon^{\frac{\gamma(\gamma-1)}{\gamma+1}}}.
	\end{equation*}
	From the centered asymptotics in \cref{lem:pointwise_intro} we obtain
	\begin{equation*}
		\frac{1}{C} \leq e^{-\lambda_m \varepsilon^{\frac{\gamma-1}{1+\gamma}} t_\varepsilon}
		\varepsilon^{\frac{\gamma(1-\gamma)}{1+\gamma}} e^{\frac{\lambda_m}{1-\gamma} \varepsilon^{\frac{\gamma(\gamma-1)}{\gamma+1}}} \leq C,
	\end{equation*}
	and hence \cref{eq:mixingtime:bound:zero} follows after taking logarithms. The fact that $d_{n,\varepsilon}(t_\varepsilon,0) \approx 1$ follows from the same argument as above.
\end{proof}

\section{{The growth of the eigenfunctions}}\label{sec:growtheigen}

In this section, we give the proofs of  \cref{lem:WKBprime}, \cref{thm:main:3prime} and \cref{cor:asymptoticsprime}.

\subsection{{WKB approximation}}\label{sec:WKBproof}
In order to obtain the asymptotic growth rate we employ a WKB type approximation for the Schrödinger equation.
First let us recall the classical transformation of the Sturm--Liouville operator to a Schrödinger equation.
That is, the equation for a given constant $\eta > 0$
\begin{equation*}
	-\eta u''(x) + V'(x) u'(x) = g(x) u(x)
\end{equation*}
satisfied in a classical sense at $x$, for $V$ twice continuously differentiable at $x$, can be transformed into the Schr\"odinger equation (a.k.a.~Witten--Laplacian)
\begin{equation} \label{eq:Witten}
	\eta^2 v'' = \left ( \frac{|V'|^2}{4} - \frac{\eta}{2} V'' - \eta g \right ) v,
\end{equation}
where $v = e^{-\frac{1}{2\eta} V} u$. Furthermore, if we know $u \in L^2(e^{-\frac{1}{\eta}V})$ then $v \in L^2(\ud x)$.

The WKB approximation is a technique coming from the asymptotic study of the eigenfunctions of the Schr\"odinger equation, which is an extension of the Liouville--Green method, see~\cite{Olver,Titchmarsh}.
The idea is to consider a power series of the log of the solution in terms of $\eta$.
It turns out that even though in our setting we are ultimately interested in $\eta=1$, this is still the relevant expansion to consider. The WKB construction stated earlier in \cref{lem:WKBprime} identifies the successive correction terms $S_n$ and the remainder $R_n$ systematically. In the specific situation relevant for this manuscript, namely when $f(x)=|x|^\gamma$ and $g \equiv \lambda$ is a constant, the later decay estimates will show that the terms $S_n$ become progressively smaller as $x\to\infty$, and that for sufficiently large $n$ the remainder is bounded. This is what allows us to recover the asymptotic growth rate of the eigenfunctions up to a multiplicative constant.

\begin{proof}[Proof of \cref{lem:WKBprime}]
	Introduce $\eta$ as a formal parameter and write
	\begin{equation*}
		v = e^{\frac{1}{\eta}\sum_{i=0}^{n} \eta^i S_i + \eta^{n} R_n}.
	\end{equation*}
	Let $f=V'$. Then the $\eta$-dependent Witten equation associated with \cref{eq:Wittenprime} reads
	\begin{equation*}
		\eta^2 v'' = \left ( \frac{f^2}{4} - \frac{\eta}{2} f' - \eta \lambda \right ) v.
	\end{equation*}
	If we write
	\begin{equation*}
		W = \frac{1}{\eta}\sum_{i=0}^{n} \eta^i S_i + \eta^{n} R_n,
	\end{equation*}
	then
	\begin{equation*}
		\eta^2\left(W'' + (W')^2\right) = \frac{f^2}{4} - \frac{\eta}{2} f' - \eta \lambda.
	\end{equation*}
	Matching powers of $\eta$ gives
	\begin{equation*}
		S_0' = -\frac{f}{2} \quad \textrm{ and }\quad S_1' = \frac{\lambda}{f},
	\end{equation*}
	where in $S_0'$ we have chosen the negative branch of the square root. For $k \geq 2$ we get the hierarchy
	\begin{equation*}
		S_{k-1}'' + \sum_{n_1+n_2=k} S_{n_1}'S_{n_2}' = 0,
	\end{equation*}
	which is an equation for $S_k'$, since the terms with $n_1=0$ or $n_2=0$ give $2S_0'S_k'=-fS_k'$. Thus,
	\begin{equation*}
		S_k' = \frac{1}{f} S_{k-1}'' + \frac{1}{f} \sum_{j=1}^{k-1} S_j'S_{k-j}'.
	\end{equation*}
	Since $S_0(x^*)=0$ by definition and $S_k(x^*)=0$ for $k\geq 2$, this yields the expansion in \cref{ec:S0S1Sk}.

	Now set $\eta=1$ and write $S=\sum_{j=0}^{n} S_j$. By the identities above,
	\begin{equation}\label{eq:cr11}
		S'' + (S')^2 = \frac{f^2}{4} - \frac{V''}{2} - \lambda + S_n'' + \sum_{n_1+n_2 \geq n+1}^{n_1 \leq n; n_2 \leq n} S_{n_1}'S_{n_2}'.
	\end{equation}
	Inserting $v=e^{S+R_n}$ into \cref{eq:Wittenprime} and using \cref{eq:cr11} yields
	\begin{equation*}
		R_n'' + 2R_n' \sum_{j=0}^{n} S_j' + (R_n')^2 = -S_n'' - \sum_{n_1+n_2 \geq n+1}^{n_1 \leq n; n_2 \leq n} S_{n_1}'S_{n_2}'.
	\end{equation*}
	Now, we write $\widehat u = \exp(R_{n})$, i.e., $R_{n} = \log(\widehat u)$ and using the chain-rule we have
	\begin{equation}\label{eq:rules}
		R_{n}' = \frac{\widehat u'}{\widehat u}, \quad
		R_{n}'' + (R_n')^2 = \frac{\widehat u''}{\widehat u}.
	\end{equation}
	Inserting \cref{eq:rules} in \cref{eq:cr11} we obtain
	\begin{equation*}
		\frac{\widehat u''}{\widehat u} + 2\frac{\widehat u'}{\widehat u} \sum_{j=0}^{n} S_j' = -S_n'' - \sum_{n_1+n_2 \geq n+1}^{n_1 \leq n; n_2 \leq n} S_{n_1}'S_{n_2}'.
	\end{equation*}
	Multiplying by $\widehat u$ gives
	\begin{equation*}
		-\widehat u'' - 2\widehat u' \sum_{j=0}^{n} S_j' = \widehat u \left ( S_n'' + \sum_{n_1+n_2 \geq n+1}^{n_1 \leq n; n_2 \leq n} S_{n_1}'S_{n_2}' \right ),
	\end{equation*}
	which is the Sturm--Liouville equation for $\widehat u$ given in \cref{eq:SLL}.
\end{proof}

\begin{lemma} \label{lem:WKB:decay}
	Let $f(x) = V'(x) = |x|^\gamma$ and $v$ be as in \cref{lem:WKBprime} for some $\gamma > 0$. Then
	for $k\geq 2$ there exists a positive constant  $C_k$ such that
	\begin{align}\label{ec.SkCk}
		|S_k'(x)|\leq C_k
		\begin{cases}
			|x|^{-(2k-1)\gamma},  & 0<\gamma\leq 1, \\
			|x|^{-(k\gamma+k-1)}, & \gamma>1,
		\end{cases}
		\quad \textrm{ for all }\quad x\geq x^*,
	\end{align}
	where $x^*$ is the corresponding turning point from \cref{lem:WKBprime}.

	When $0<\gamma\leq 1$ and $1/\gamma$ is an odd integer, that is $1-(2k-1)\gamma = 0$, it follows that
	\begin{equation}\label{eq:Sklog}
		|S_k(x)| \leq  C_k\log(x) \quad \textrm{ for all }\quad x\geq x^*.
	\end{equation}

	In particular, whenever $n > \frac{1-\gamma}{2\gamma}$ the quantity
	\begin{align}\label{eq:SSs}
		\frac{S_n''+\sum_{n_1+n_2 \geq n+1}^{n_1 \leq n; n_2 \leq n} S_{n_1}'S_{n_2}'}{\sum_{j=0}^{n} S_j'} \in L^1(x^*,\infty).
	\end{align}
\end{lemma}
\begin{proof} For any $k$
	we note  that leading order term of  $S_k'$  is a term of the form $a_k x^{q_k}$ for some $a_k>0$ and $q_k \in \R$.

	By the definition of $S_2$, we
	first note that $S_2'$ satisfies this condition and let us make the induction assumption that \cref{ec.SkCk} is true for every $2 \leq j < k$. Then using \cref{ec:S0S1Sk} and our induction assumption, we get if $0<\gamma\leq 1$, then
	\begin{align*}
		|S_k'(x)|
		 & \lesssim
		|x|^{-\gamma}|x|^{-((2k-3)\gamma+1)}
		+ |x|^{-\gamma}\sum_{j=1}^{k-1}|x|^{-(2j-1)\gamma}|x|^{-(2(k-j)-1)\gamma}
		\\
		 & \lesssim
		|x|^{-((2k-2)\gamma+1)}
		+ |x|^{-(2k-1)\gamma}
		\lesssim
		|x|^{-(2k-1)\gamma}.
	\end{align*}

	If $\gamma>1$, then
	\begin{align*}
		|S_k'(x)|
		 & \lesssim
		|x|^{-\gamma}|x|^{-((k-1)\gamma+k-1)}
		+ |x|^{-\gamma}\sum_{j=1}^{k-1}|x|^{-(j\gamma+j-1)}|x|^{-((k-j)\gamma+(k-j)-1)}
		\\
		 & \lesssim
		|x|^{-(k\gamma+k-1)}
		+ |x|^{-((k+1)\gamma+k-2)}
		\lesssim
		|x|^{-(k\gamma+k-1)}.
	\end{align*}

	This completes the proof of \cref{ec.SkCk}.

	If $0<\gamma\leq 1$ and $1-(2k-1)\gamma = 0$, then \cref{ec.SkCk} yields $|S_k'(x)|\lesssim x^{-1}$, and integrating from $x^*$ to $x$ gives \cref{eq:Sklog}.

	To prove \cref{eq:SSs} we note that from \cref{ec:S0S1Sk,ec.SkCk}, it is enough to prove that
	\begin{align*}
		\frac{|S_n''|}{|S_0'|} \in L^1(x^*,\infty)
	\end{align*}
	and
	\begin{align*}
		\frac{|S_{n_1}'||S_{n_2}'|}{|S_0'|} \in L^1(x^*,\infty)\quad \textrm{ for } \quad n_1 + n_2 = n+1.
	\end{align*}
	If $0<\gamma\leq 1$, then using \cref{ec.SkCk} we get
	\begin{align*}
		\frac{|S_n''(x)|}{|S_0'|}
		 & \lesssim
		\frac{|x|^{-((2n-1)\gamma+1)}}{|x|^{\gamma}}
		=
		|x|^{-(2n\gamma+1)},
		\\
		\frac{|S_{n_1}'(x)||S_{n_2}'(x)|}{|S_0'|}
		 & \lesssim
		\frac{|x|^{-2n\gamma}}{|x|^{\gamma}}
		=
		|x|^{-(2n+1)\gamma}.
	\end{align*}
	If $\gamma>1$, then using \cref{ec.SkCk} we get
	\begin{align*}
		\frac{|S_n''(x)|}{|S_0'|}
		 & \lesssim
		\frac{|x|^{-(n\gamma+n)}}{|x|^{\gamma}}
		=
		|x|^{-((n+1)\gamma+n)},
		\\
		\frac{|S_{n_1}'(x)||S_{n_2}'(x)|}{|S_0'|}
		 & \lesssim
		\frac{|x|^{-((n_1+n_2)\gamma+(n_1+n_2)-2)}}{|x|^{\gamma}}
		\lesssim
		|x|^{-((n+2)\gamma+n-1)}.
	\end{align*}
	Taking $n > \frac{1-\gamma}{2\gamma}$ we conclude \cref{eq:SSs}.
\end{proof}

\subsection{Growth of the error term \texorpdfstring{$R_n$}{Rn}}
\label{sec:growth}

In this subsection, we deal with the error term $R_n$ in the WKB approximation.
In order to achieve this we will actually develop a rather general method for getting the highest order growth of a solution to a Sturm--Liouville equation. The proof relies on the fact that for a solution to the following Sturm--Liouville equation
\begin{align}\label{ec:STODE}
	-u''(x) + f(x) u'(x) = g(x) u(x)
\end{align}
the second order term will not influence the growth in the regions where $f$ is much bigger than $g$ and the dominating growth will be that of the first order ordinary differential equation (for short ODE)
\begin{align*}
	u' = \frac{g}{f} u,\quad
	\textrm{ that is, }\quad u = C \exp \left (\int \frac{g}{f} \right ).
\end{align*}
The next lemma gives us the necessary reduction: we write \cref{ec:STODE} as a first order system and express its solution in polar coordinates. The angle variable then satisfies a first-order ODE that is \textit{decoupled} from the radial variable. Thus, the entire problem reduces to building barriers for that first order ODE.

\begin{lemma} \label{lem:ODE}
	Let $f \in \mathcal{C}(\R_+)$ and $g \in \mathcal{C}(\R_+)$. Let $u \in \mathcal{C}^2(\R_+)$ be a solution to
	\begin{equation}\label{eq:g:ODE}
		-u'' + f u' = g u\quad \textrm{ in }\quad \R_+
	\end{equation}
	with boundary condition $(u(0),u'(0)) = X_0\in \R^2$.
	For all $t\geq 0$
	denote $X(t):= (u(t),u'(t))$ and define the radial component $r$ and angular component $Y$ such that $X(t) = r(t) Y(t)$ with $Y(t): = (\cos(\theta(t)), \sin(\theta(t)))$. Then $r \in \mathcal{C}^1(\R^+)$ and satisfies the non-homogeneous ODE
	\begin{equation} \label{eq:ODE:r}
		r'(t) = r(t) A(t) Y(t) \cdot Y(t), \quad \text{where} \quad A(t) =
		\begin{bmatrix}
			0     & 1    \\
			-g(t) & f(t)
		\end{bmatrix},
	\end{equation}
	where $\cdot$ denotes the usual inner product in $\mathbb{R}^2$.
	Furthermore, $\theta \in \mathcal{C}^1(\R_+)$ satisfies the ODE
	\begin{equation} \label{eq:ODE:theta}
		\theta'(t) = -1 - (g(t)-1) \cos^2(\theta(t)) + \frac{f(t)}{2} \sin(2\theta(t)).
	\end{equation}
\end{lemma}

\begin{remark}
	It is not hard to see that
	the ODE \cref{eq:g:ODE} can be expressed as two-dimensional ODE
	\begin{equation} \label{eq:ODE:system}
		X' = A X,
	\end{equation}
	where $X(t)= (u(t),u'(t))^T$ and
	$A$ is the matrix in \cref{eq:ODE:r}. Nevertheless,   computing the explicit solution of \cref{eq:ODE:system} is hard  due to the  non-homogeneity.
\end{remark}
\begin{proof}[Proof of \cref{lem:ODE}]
	We start writing  $X \cdot X = r^2$ and differentiate to get
	\begin{equation*}
		(r^2)' =  (X \cdot X)' = 2 A X \cdot X = 2 r^2 A Y \cdot Y,
	\end{equation*}
	which gives \cref{eq:ODE:r}.
	To find the ODE for $\theta$ we first note that
	\begin{equation*}
		X' = r' Y + r Y' = r (A Y \cdot Y) Y + r \theta' Y^\perp,
	\end{equation*}
	where $Y^\perp = (-\sin(\theta),\cos(\theta))$.
	Using \cref{eq:ODE:system} and taking the dot product of the above with $Y^\perp$ gives
	\begin{equation*}
		r \theta'=A X \cdot Y^\perp =r A Y \cdot Y^\perp,
	\end{equation*}
	which gives
	\begin{equation*}
		\theta' = A Y \cdot Y^\perp.
	\end{equation*}
	Rewriting the above explicitly completes the proof of \cref{eq:ODE:theta}.
\end{proof}

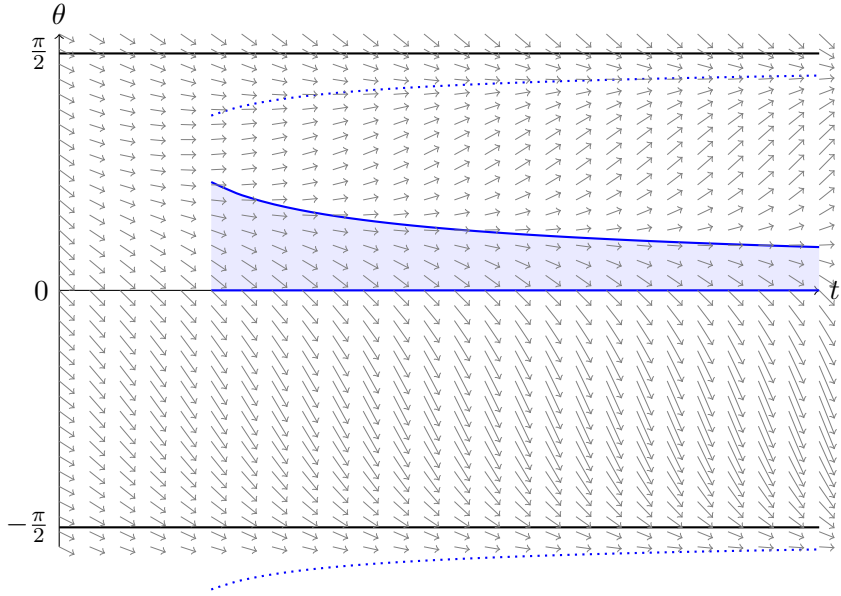
\begin{figure}
	\centering
	\ifdraft{
		\begin{tikzpicture}
			\path[use as bounding box] (0,0) rectangle (5cm,5cm);
		\end{tikzpicture}
	}{
		\begin{tikzpicture}[x=2cm,y=2cm]
			\draw[->] (0,0) -- (5,0) node[right] {$t$};
			\draw[->] (0,-1.7) -- (0,1.7) node[above] {$\theta$};
			\node[left] at (0,1.57) {$\tfrac{\pi}{2}$};
			\node[left] at (0,-1.57) {$-\tfrac{\pi}{2}$};
			\node[left] at (0,0) {$0$};
			\def\fx{(10*\x)^(1/2)}
			\def\gx{2}
			\def\tomte{atan((\fx+sqrt(\fx^2-4*\gx))/2)*pi/180}
			\draw[thick, blue, dotted, domain=1:5, smooth, variable=\x]
			plot ({\x},{\tomte});
			\draw[thick, blue, dotted, domain=1:5, smooth, variable=\x]
			plot ({\x},{-pi/2-asin((3*sqrt(10*\x)-sqrt(10*\x-8))/(10*\x+1))*pi/360});
			\def\upperexpr{atan((\fx-sqrt(\fx^2-4*\gx))/2)*pi/180}
			\def\lowerexpr{\x*0}			\fill[blue!20,opacity=0.4,smooth,domain=1:5]
			plot ({\x},{\upperexpr})
			-- plot[domain=5:1] ({\x},{\lowerexpr})
			-- cycle;
			\draw[thick, blue, domain=1:5, smooth] plot ({\x},{\upperexpr});
			\draw[thick, blue, domain=1:5, smooth] plot ({\x},{\lowerexpr});
			\draw[thick] (0,1.5707963268) -- (5,1.5707963268);
			\draw[thick] (0,-1.5707963268) -- (5,-1.5707963268);
			\foreach \x in {0,0.2,...,5} {
					\foreach \y in {-1.7,-1.6,...,1.7} {
							\draw[->,gray,very thin] (\x,\y) -- ++(0.1,{0.05*(-1-(2-1)*((cos(\y r))^2)+((\x*10)^(1/2)/2)*((sin((2*\y) r))))});
						}
				}
		\end{tikzpicture}}
	\caption{The vector field for the angle ODE \cref{eq:ODE:theta} when $\gamma=1/2$ and $g=2$. The blue shaded region depicts the area between $[\theta_l=0,\theta_u]$, and the dotted curve are the asymptotes for the limits $\pi/2$ and $-\pi/2$. The $\omega$-limit is expected to be the set $\{-\pi/2,0,\pi/2\}$.}
	\label{fig:angleODE}
\end{figure}

The next proposition makes the phase portrait in Figure~\ref{fig:angleODE} precise. It provides barrier curves $\theta_l$ and $\theta_u$ that separate the trajectories converging to $-\pi/2$ and $\pi/2$, respectively, and in particular implies that the $\omega$-limit set of any solution of \cref{eq:ODE:theta} is contained in $\{-\pi/2,0,\pi/2\}$.

To control the radial component, we next rewrite the coefficient in the equation for $r$ in terms of $\theta$. This identity will be used in the next two results: first to show that $r$ remains bounded along trajectories with $\theta(t)\to 0$, and then, under the additional $L^2$ assumption, to rule out the limits $\pm\pi/2$ and conclude that $\theta(t)\to 0$.
By the definition of $A$ and $Y$ we obtain
\begin{equation}\label{eq:AYY}
	A Y \cdot Y = \sin(\theta)\cos(\theta) (1-g) + f \sin^2(\theta),
\end{equation}
where $\theta$ solves \cref{eq:ODE:theta}.

\begin{proposition} \label{lem:theta:convergence}
	Let $f$ and $g$ be as in \cref{lem:ODE}.
	Assume that $f^2(t)\geq 4g(t)$ for all $t\geq 0$, that $f(t) \to \infty$ and $g(t)/f(t) \to 0$ as $t \to \infty$.
	Then there exists a time $T>0$, an increasing function $\theta_l:[T,\infty)\to(-\pi/2,0]$ and a decreasing function $\theta_u:[T,\infty)\to[0,\pi/2)$ with
	\begin{equation*}
		\theta_l(t)\to0 \quad \text{and} \quad \theta_u(t)\to0 \quad \text{as } t\to\infty,
	\end{equation*}
	such that:

	\begin{itemize}
		\item[(a)] For each $t_0\ge T$, any solution $\theta$ to \cref{eq:ODE:theta} satisfying
		      \begin{equation*}
			      \theta(t_0)\in[\theta_u(t_0),\pi/2]
		      \end{equation*}
		      converges to $\pi/2$ as $t\to\infty$.

		\item[(b)] For each $t_0\ge T$, any solution $\theta$ to \cref{eq:ODE:theta} satisfying
		      \begin{equation*}
			      \theta(t_0)\in[-\pi/2,\theta_l(t_0)]
		      \end{equation*}
		      converges to $-\pi/2$ as $t\to\infty$.
	\end{itemize}

	Moreover, $\theta_u$ and $\theta_l$ can be chosen explicitly as
	\begin{equation}\label{eq:def-theta-u}
		\theta_u(t)
		=
		\sup_{s\ge t} \arctan\!\left(
		\frac{f(s) - \sqrt{f^2(s)-4g_+(s)}}{2}
		\right),
		\qquad t\ge T,
	\end{equation}
	and
	\begin{equation}\label{eq:def-theta-l}
		\theta_l(t)
		=
		\inf_{s\ge t} \arctan\!\left(
		\frac{f(s) - \sqrt{f^2(s)+4g_-(s)}}{2}
		\right),
		\qquad t\ge T,
	\end{equation}
	where $g_+:=\max\{g,0\}$ and $g_-:=\max\{-g,0\}$.
	In particular, for any solution $\theta$ of \cref{eq:ODE:theta}, every accumulation point of the trajectory
	$(\theta(t))_{t\geq 0}$ belongs to the set $\{-\pi/2,0,\pi/2\}$.
\end{proposition}

\begin{proof}
	Let $G(\vartheta,t):=-1 - (g(t)-1) \cos^2(\vartheta) + \frac{f(t)}{2} \sin(2\vartheta)$ for $t\geq 0$ and $\vartheta\in [-\pi,\pi]$.
	Note that $G(\pi/2,t) = -1$ for all $t\geq 0$.

	\smallskip

	\noindent
	\textbf{Step 1: Stationary points of $G(\cdot,t)$:}
	Since $G(\cdot,t)$ has period $\pi$ it is enough to consider $\vartheta \in [-\pi/2,\pi/2]$.
	Make the change of variables $x = \tan(\vartheta)$. Then
	\begin{equation*}
		\cos^2(\vartheta) = \frac{1}{1+x^2}, \quad \sin(2\vartheta) = \frac{2x}{1+x^2}.
	\end{equation*}
	Substituting into $G(\vartheta,t)=0$ and multiplying by $1+x^2$ gives
	\begin{equation*}
		x^2-x f(t) + g(t) = 0.
	\end{equation*}
	The roots are given by
	\begin{equation*}
		x_\pm(t)= \frac{f(t) \pm \sqrt{f^2(t)-4g(t)}}{2}.
	\end{equation*}
	Since $f(t) \to \infty$ as $t \to \infty$ and $g(t)/f(t) \to 0$ as $t \to \infty$, we have
	\begin{equation}\label{eq:xplusminus}
		\lim_{t \to \infty} x_+(t) = \infty, \quad \lim_{t \to \infty} x_-(t) = 0.
	\end{equation}
	Translating back to $\vartheta$ using $\vartheta = \arctan(x)$ we obtain the two zeros of $G(\cdot,t)$ in $[-\pi/2,\pi/2]$. Denote
	\begin{equation*}
		\theta^\ast(t):=\arctan(x_+(t)).
	\end{equation*}
	Then \cref{eq:xplusminus} gives $\theta^\ast(t)\to \pi/2$ as $t\to\infty$.
	See \cref{fig:angleODE} for an illustration of the vector field defined by $G$.

	\smallskip

	\noindent\textbf{Step 2: Barrier:}

	For $t$ large enough we have
	\begin{equation*}
		\theta_l(t)\leq \arctan(x_-(t))\leq \theta_u(t)\leq \theta^\ast(t).
	\end{equation*}
	Since
	\begin{equation*}
		G(\vartheta,t)= -\frac{x^2-f(t)x+g(t)}{1+x^2}, \qquad x=\tan(\vartheta),
	\end{equation*}
	the sign of $G(\cdot,t)$ is determined by the quadratic $x^2-f(t)x+g(t)$, it follows that for any $\theta_u(t)<\theta<\theta^\ast(t)$ we have
	\begin{equation*}
		G(\theta,t)>0,
	\end{equation*}
	so any solution $\theta$ to \cref{eq:ODE:theta} that is in the interval $[\theta_u(t),\pi/2]$ for some $t \geq T$ converges to $\pi/2$ as $t \to \infty$.

	Likewise, for any $-\pi/2 \le \theta < \theta_l(t)$ we have
	\begin{equation*}
		G(\theta,t)<0,
	\end{equation*}
	so any solution $\theta$ to \cref{eq:ODE:theta} that is in the interval $[-\pi/2,\theta_l(t)]$ for some $t \geq T$ converges to $-\pi/2$ as $t \to \infty$.

	Finally, if a solution $\theta$ has an accumulation point in $(-\pi/2,\pi/2)\setminus\{0\}$, then for $t$ large enough it stays either strictly above $\theta_u(t)$ or strictly below $\theta_l(t)$, which contradicts the two alternatives just proved. Therefore, every accumulation point belongs to $\{-\pi/2,0,\pi/2\}$.
\end{proof}

\begin{corollary}[The growth of $r$]\label{cor:growthr}
	\hfill

	\noindent
	Assume the hypotheses of \cref{lem:theta:convergence}, and in addition that
	$\nicefrac{g}{f}\in L^1([T,\infty))$ and $g\in L^\infty([T,\infty))$ for the time $T>0$
	appearing in \cref{lem:theta:convergence}.
	Let $\theta$ be a solution of \cref{eq:ODE:theta} that satisfies
	$\theta(t)\to 0$ as $t\to\infty$, and let $r$ be the corresponding radial
	component solving \cref{eq:ODE:r}. Then $r$ is bounded on $[T,\infty)$.
\end{corollary}
\begin{proof}
	By \cref{eq:ODE:r} and \cref{eq:AYY} we have, for $t\ge T$,
	\begin{equation}\label{eq:repr-r}
		r(t)
		=
		r(T)\exp\!\left(\int_T^t A(s)Y(s)\cdot Y(s)\,\mathrm ds\right)
		=
		r(T)\exp\!\left(\int_T^t
		\bigl[\sin\theta\,\cos\theta\,(1-g)+f\,\sin^2\theta\bigr](s)\,\mathrm ds\right),
	\end{equation}
	where $Y(s)=(\cos\theta(s),\sin\theta(s))$.

	Since $\theta(t)\to0$ as $t\to\infty$, there exists $T_1\ge T$ such that
	$|\theta(t)|\le1$ for all $t\ge T_1$. On $[T_1,\infty)$ we use the standard bounds
	\begin{equation*}
		|\sin\theta|\le|\theta|,\qquad
		|\cos\theta|\le 1,\qquad
		\sin^2\theta\le \theta^2.
	\end{equation*}
	Hence, for $t\ge T_1$,
	\begin{align}
		\bigl\lvert A(t)Y(t)\cdot Y(t)\bigr\rvert
		 & =
		\bigl\lvert \sin\theta(t)\,\cos\theta(t)\,\bigl(1-g(t)\bigr)
		+ f(t)\,\sin^2\theta(t)\bigr\rvert
		\notag \\
		 & \le
		|\theta(t)|\,\bigl(1+|g(t)|\bigr)
		+ |f(t)|\,\theta^2(t).
		\label{eq:AYY-est}
	\end{align}

	By \cref{lem:theta:convergence} and its construction of the barriers $\theta_l$ and $\theta_u$,
	for $t$ large enough we have
	\begin{equation*}
		|\theta(t)| \,\le\, C\,\frac{|g(t)|}{f(t)}
	\end{equation*}
	for some constant $C>0$ independent of $t$ (this follows from the explicit
	formula for $\theta_l,\ \theta_u$ in \cref{eq:def-theta-u} and the assumptions
	$f(t)\to\infty$ and $g(t)/f(t)\to0$).
	Plugging this into \cref{eq:AYY-est} yields, for $t\ge T$, since $g \in L^\infty([T,\infty))$,
	\begin{align*}
		\bigl\lvert A(t)Y(t)\cdot Y(t)\bigr\rvert
		 & \le
		C\,\frac{|g(t)|}{f(t)}\,(1+|g(t)|)
		+ C^2\,\frac{|g(t)|^2}{f(t)} \\
		 & \le
		C'\,\frac{|g(t)|}{f(t)}
	\end{align*}
	for some constant $C'>0$ independent of $t$.
	Since $\nicefrac{g}{f} \in L^1([T,\infty))$ by assumption, we have
	\begin{equation*}
		\int_{T}^\infty \bigl\lvert A(s)Y(s)\cdot Y(s)\bigr\rvert\,\mathrm ds < \infty.
	\end{equation*}
	Therefore, the exponential in \cref{eq:repr-r} remains bounded for $t\ge T$.
\end{proof}

\begin{proposition} \label{prop:theta:convergence}
	Assume the hypotheses of \cref{lem:theta:convergence}.
	In addition, let $g$ and $f$ be as in \cref{cor:growthr}, that is,
	$\nicefrac{g}{f} \in L^1([T,\infty))$ and $g \in L^\infty([T,\infty))$.
	Let
	$u$ be a solution to \cref{eq:g:ODE} such that $u \in L^2_{\varrho}$, where $\varrho(t):= e^{-\int_0^t f(s) \ds}$, $t\geq 0$. Then it follows that $\theta(t) \to 0$ as $t \to \infty$.
\end{proposition}
\begin{proof}
	By~\cref{lem:theta:convergence} we have that $\theta(t) \to \theta_*$ as $t \to \infty$, where $\theta_*\in \{-\pi/2,0,\pi/2\}$.
	The proof is done by a contradiction argument.
	Assume that $\theta_*=\pi/2$. Then it follows that  $\pi/2-\theta(t) = \kappa(t)$ for some $\kappa(t) > 0$ and $t > T$, and we have
	\begin{align*}
		A(t) Y(t) \cdot Y(t)
		 & =
		\begin{bmatrix}
			0     & 1    \\
			-g(t) & f(t)
		\end{bmatrix}
		\begin{bmatrix}
			\sin(\kappa(t)) \\
			\cos(\kappa(t))
		\end{bmatrix}
		\cdot
		\begin{bmatrix}
			\sin(\kappa(t)) \\
			\cos(\kappa(t))
		\end{bmatrix}
		\\
		 & =
		\begin{bmatrix}
			\cos(\kappa(t)) \\
			-g(t) \sin(\kappa(t)) + f(t) \cos(\kappa(t))
		\end{bmatrix}
		\cdot
		\begin{bmatrix}
			\sin(\kappa(t)) \\
			\cos(\kappa(t))
		\end{bmatrix}
		\\
		 & =
		\cos(\kappa(t))\sin(\kappa(t)) - g(t) \sin(\kappa(t)) \cos(\kappa(t)) + f(t) \cos^2(\kappa(t)) \\
		 &
		\geq
		(1-c(\kappa(t))) f(t),
	\end{align*}
	where $Y(t) = (\cos(\theta(t)), \sin(\theta(t)))^{T}$ and  $c(\kappa(t)) \to 0$ as $t \to \infty$.  Assume that $T$ is so large that $c(\kappa(t)) < 1/8$ for $t \geq T$. Since $r$ solves \cref{eq:ODE:r},  for $t > T$
	\begin{align*}
		r(t) \geq r(T) \exp \left ( \int_{T}^t \frac{7}{8} f(s) \ud s \right )\quad \textrm{ and }\quad \sin(\theta(t))\geq\frac{1}{2}.
	\end{align*}
	Recall the notation of \cref{lem:ODE}, that is,  $u(t)= r(t) \cos(\theta(t))$ and $u'(t)= r(t)\sin(\theta(t))$. Thus, our above analysis means that for $t > T$ we have the estimate
	\begin{align*}
		u'(t) = r(t) \sin(\theta(t)) \geq \frac{1}{2} r(T) \exp \left ( \int_{T}^t \frac{7}{8}  f(s) \ud s \right ).
	\end{align*}
	Integrating over $t>T$ we get the following lower bound
	\begin{align*}
		u(t) - u(T) \geq \frac{1}{2} r(T) \int_{T}^t \exp \left ( \int_{T}^s \frac{7}{8}  f(\tau) \ud \tau \right ).
	\end{align*}
	For $u$ to be in $L^2_\varrho$ we need at least $\sqrt{\varrho(t)} u(t) \to 0$ as $t \to \infty$, specifically, we need
	\begin{align*}
		e^{-\frac{1}{2}\int_T^t f(\tau) \ud \tau} \int_{T}^t \exp \left ( \int_{T}^s \frac{7}{8}  f(q) dq \right )\ud s  \to 0,\quad \textrm{ as }\quad t\to \infty,
	\end{align*}
	which is a contradiction since $f(t) \to \infty$ as $t \to \infty$. The case when $\theta(t) \to -\pi/2$ is similar.

	In conclusion, from \cref{lem:theta:convergence} we have that $\theta(t) \to 0$ as $t \to \infty$.
\end{proof}

\begin{corollary}\label{cor:Rn-bounded}
	Under the hypotheses of \cref{lem:WKBprime}, assume in addition that $f(x)=|x|^{\gamma}$ with $\gamma>0$, $\lambda>0$, $v \in L^2([x^\ast,\infty))$, and choose any integer $n>\frac{1+\gamma}{2\gamma}$. Then the remainder $R_n$ in the WKB expansion satisfies
	\begin{equation*}
		\sup_{x\ge x^*} \bigl|R_n(x)\bigr| < \infty.
	\end{equation*}
	Consequently, for $x\ge x^*$ one has
	\begin{equation*}
		v(x) \approx \exp\Bigl(\sum_{k=0}^{n} S_k(x)\Bigr),
	\end{equation*}
	with implied constants depending only on $\gamma,\lambda,n$ and $x^*$.
\end{corollary}
\begin{proof}
	From \cref{lem:WKBprime} we have that $\widehat u:=\exp(R_n)$ solves \cref{ec:STODE} with drift coefficient
	\begin{equation*}
		\widetilde f(x) = -2\sum_{j=0}^{n} S_j'(x) = |x|^{\gamma} + \mathsf{O}(1) \quad \text{as } x \to \infty,
	\end{equation*}
	and forcing term
	\begin{equation*}
		\widetilde g(x) = S_n''(x) + \sum_{n_1+n_2 \geq n+1}^{n_1 \leq n; n_2 \leq n} S_{n_1}'(x)S_{n_2}'(x).
	\end{equation*}
	Since $n>\frac{1+\gamma}{2\gamma}$ implies $n>\frac{1-\gamma}{2\gamma}$, \cref{lem:WKB:decay} yields $\widetilde g/\widetilde f \in L^1([x^*,\infty))$. The same lemma shows that each term in $\widetilde g$ is bounded on $[x^*,\infty)$, hence $\widetilde g\in L^\infty([x^*,\infty))$. Moreover, $\widetilde f(x)\to\infty$ and $\widetilde f(x)^2\geq 4|\widetilde g(x)|$ for $x$ large enough.

	To apply \cref{prop:theta:convergence}, define
	\begin{equation*}
		\varrho(x):=\exp\left(-\int_{x^*}^x \widetilde f(s)\,\mathrm ds\right).
	\end{equation*}
	Since $\widetilde f=-2\sum_{j=0}^n S_j'$, we have
	\begin{equation*}
		\varrho(x)=C\exp\left(2\sum_{j=0}^n S_j(x)\right)
	\end{equation*}
	for some constant $C>0$, and therefore
	\begin{equation*}
		\varrho(x)\widehat u(x)^2
		=
		C\exp\left(2\sum_{j=0}^n S_j(x)+2R_n(x)\right)
		=
		C v(x)^2.
	\end{equation*}
	Because $v\in L^2([x^*,\infty))$, it follows that $\widehat u\in L^2_{\varrho}(x^*,\infty)$.

	Hence, \cref{prop:theta:convergence} and \cref{cor:growthr} apply after enlarging $x^*$ if necessary. In particular, $\theta(x)\to 0$ and the proof of \cref{cor:growthr} shows that
	\begin{equation*}
		\int_{x^*}^{\infty} \bigl\lvert A(s)Y(s)\cdot Y(s)\bigr\rvert\,\mathrm ds < \infty.
	\end{equation*}
	Using the representation formula for $r$ from \cref{eq:repr-r}, we obtain constants $0<c<C<\infty$ such that
	\begin{equation*}
		c\leq r(x)\leq C \qquad \text{for all } x\geq x^*.
	\end{equation*}
	Since $\theta(x)\to 0$, after increasing $x^*$ once more if necessary we also have $\cos\theta(x)\geq \frac12$ for all $x\geq x^*$. Therefore,
	\begin{equation*}
		\frac{c}{2}\leq \widehat u(x)=r(x)\cos\theta(x)\leq C
		\qquad \text{for all } x\geq x^*,
	\end{equation*}
	so $R_n(x)=\log \widehat u(x)$ is bounded on $[x^*,\infty)$. The boundedness on the remaining compact interval is immediate.
\end{proof}

\begin{proof}[Proof of \cref{thm:main:3prime}]
	Set
	\begin{equation*}
		v(x):=e^{-\frac{1}{2}V(x)}u(x).
	\end{equation*}
	Then $v\in L^2([x^\ast,\infty))$ and, by the reduction carried out before \cref{lem:WKBprime}, $v$ solves \cref{eq:Wittenprime} on $[x^\ast,\infty)$. Hence, \cref{cor:Rn-bounded} gives that $R_n$ is bounded on $[x^\ast,\infty)$, and consequently
	\begin{equation*}
		v(x)\approx \exp\Bigl(\sum_{j=0}^n S_j(x)\Bigr), \qquad x\ge x^\ast.
	\end{equation*}
	Since $S_0(x)=-\frac{1}{2}V(x)$ by \cref{ec:S0S1Sk}, multiplying by $e^{V(x)/2}$ yields
	\begin{equation*}
		u(x)=e^{V(x)/2}v(x)\approx \exp\Bigl(\sum_{j=1}^n S_j(x)\Bigr), \qquad x\ge x^\ast,
	\end{equation*}
	which proves the first claim.

	For the explicit asymptotics, we isolate in each $S_k$ its leading-order contribution and use \cref{lem:WKB:decay} to control the rest. The case $k=1$ is explicit:
	\begin{equation*}
		S_1'(x)=\lambda |x|^{-\gamma},
	\end{equation*}
	and therefore
	\begin{equation*}
		S_1(x)=\lambda \frac{|x|^{1-\gamma}}{1-\gamma}+\mathsf{O}(1)
	\end{equation*}
	for $\gamma\neq 1$, while for $\gamma=1$ one gets a logarithmic term. If $\gamma>1$, then $S_1$ is bounded and the recurrence \cref{ec:S0S1Sk} shows inductively that $S_k'$ is integrable at infinity for every $k\geq 2$, hence all $S_k$ with $k\geq 2$ are bounded as well. In that case the stated asymptotic follows after absorbing those bounded terms into the comparison constants. It therefore remains to consider $0<\gamma\leq 1$. For $k\ge2$, write the recurrence \cref{ec:S0S1Sk} as
	\begin{equation*}
		S_k'(x)=\frac{1}{|x|^\gamma}S_{k-1}''(x)+\frac{1}{|x|^\gamma}\sum_{j=1}^{k-1}S_j'(x)S_{k-j}'(x).
	\end{equation*}
	Using the induction hypothesis for $S_1',\dots,S_{k-1}'$, the product of the leading terms contributes a multiple of $|x|^{-(2k-1)\gamma}$. Every remaining term contains either a derivative or at least one lower-order remainder, and by \cref{lem:WKB:decay} all such contributions are integrable at infinity after subtracting that leading term (see \cref{sec:menor} for explicit calculations). Hence, there exists a constant $C_k\in\R$ such that
	\begin{equation*}
		S_k'(x)=C_k |x|^{-(2k-1)\gamma}+r_k(x),
	\end{equation*}
	where $r_k\in L^1([x^\ast,\infty))$ if $1-(2k-1)\gamma\neq 0$, while in the resonant case $1-(2k-1)\gamma=0$ we have $|r_k(x)|\lesssim x^{-1-\eta}$ for some $\eta>0$. Integrating from $x^\ast$ to $x$ yields
	\begin{equation*}
		S_k(x)=
		\begin{cases}
			C_k |x|^{1-(2k-1)\gamma}+\mathsf{O}(1), & 1-(2k-1)\gamma\neq 0,
			\\
			C_k \log|x|+\mathsf{O}(1),              & 1-(2k-1)\gamma=0.
		\end{cases}
	\end{equation*}
	If $\gamma^{-1}\notin 2\mathbb{N}-1$, then summing these leading terms for $k=1,\dots,n$ gives the first explicit formula in the statement. If instead $\gamma = \frac{1}{2m-1}$ for some $m\geq 2$, then the logarithmic contribution occurs at the index $k=m$, while the terms with $k\geq m+1$ are bounded and can be absorbed into the comparison constants. This gives the second explicit formula.
\end{proof}

\begin{proof}[Proof of \cref{cor:asymptoticsprime}]
	Define the rescaled function
	\begin{equation*}
		u(y):=v\bigl(\varepsilon^{\frac{1}{1+\gamma}}y\bigr), \qquad y\in\R.
	\end{equation*}
	By \cref{lem:g:scaling}, $u\in L^2_{\varrho_{1,\infty}}(\R)$ and it solves
	\begin{equation*}
		-u''+V'u'=\lambda u \qquad \text{on } \R.
	\end{equation*}
	Since $\frac{1}{3}<\gamma<1$, we may take $n=2$ in \cref{thm:main:3prime}. The remark following that theorem gives
	\begin{equation*}
		S_1(y)=\lambda\frac{|y|^{1-\gamma}}{1-\gamma}+C_1,
		\qquad
		S_2(y)=\frac{\lambda}{2} |y|^{-2\gamma}+\frac{\lambda^2}{1-3\gamma}|y|^{1-3\gamma}+C_2,
	\end{equation*}
	so $S_2$ is bounded on $[x^\ast,\infty)$; the boundedness of the remainder $R_2$ follows from \cref{cor:Rn-bounded}. Therefore,
	\begin{equation*}
		u(y)\approx e^{S_1(y)}=\exp\Bigl(\lambda\frac{|y|^{1-\gamma}}{1-\gamma}\Bigr), \qquad y\ge x^\ast.
	\end{equation*}
	Substituting $y=x\varepsilon^{-\frac{1}{1+\gamma}}$ gives
	\begin{equation*}
		v(x)\approx \exp\Bigl(\lambda\frac{|x\varepsilon^{-\frac{1}{1+\gamma}}|^{1-\gamma}}{1-\gamma}\Bigr), \qquad x\in [x^\ast\varepsilon^{\frac{1}{1+\gamma}},\infty),
	\end{equation*}
	as claimed.
\end{proof}

\appendix

\section{Spectral representation of the \texorpdfstring{$\chi^2$}{chi2}--distance to equilibrium}\label{sec:srdist}
In this section,  for completeness we provide the classical spectral representation of the $\chi^2$--distance and hence for the Galerkin projection distance \cref{eq:Ga}.

Let $\varrho_{\varepsilon,\infty}(x) = C_\varepsilon e^{-V(x)/\varepsilon}$, $x\in \mathbb{R}$, be the stationary density of \cref{eq:FP}. Then we can define the
$\chi^2$-discrepancy between a density $\varrho$ and the equilibrium density $\varrho_{\varepsilon,\infty}$ as
\begin{align}\label{eq:Sepdist}
	d_\varepsilon (\varrho):= \left ( \int_{\mathbb{R}} |\varrho(x) - \varrho_{\varepsilon,\infty}(x)|^2 \varrho^{-1}_{\varepsilon,\infty}(x) \dx\right )^{1/2}.
\end{align}
Here, we denote $d_\varepsilon(\varrho)=\infty$ if the integral does not converge, and we see that $d_\varepsilon (\varrho)\in [0,\infty]$. This distance is natural with respect to the spectrum of the generator \cref{eq:FP} as we see in the following elementary lemma.

\begin{lemma}[Eigen-expansion of $d_\varepsilon(\varrho)$]\label{lem:separation}\hfill

	\noindent
	Let $\varepsilon>0$ be fixed.
	Let $\psi_{k,\varepsilon}$, $k=0,1,\ldots$ be an orthonormal basis of $L^2_{\varrho_{\varepsilon,\infty}}$ with $\psi_{0,\varepsilon}=1$. Then for a density $\varrho \in L^2_{\varrho^{-1}_{\varepsilon,\infty}}$, the $\chi^2$-discrepancy \cref{eq:Sepdist} satisfies
	\begin{equation*}
		(d_\varepsilon(\varrho))^2 = \sum_{k=1}^\infty \left ( \int \varrho(x) \psi_{k,\varepsilon}(x) \dx \right )^2.
	\end{equation*}
\end{lemma}
\begin{proof}
	We begin by rewriting
	\begin{align*}
		(d_\varepsilon(\varrho))^2 = \int \left |\frac{\varrho(x) - \varrho_{\varepsilon,\infty}(x)}{\varrho_{\varepsilon,\infty}(x)} \right |^2 \varrho_{\varepsilon,\infty}(x) \dx.
	\end{align*}
	Then expand the function $u = \frac{\varrho - \varrho_{\varepsilon,\infty}}{\varrho_{\varepsilon,\infty}}$ using our basis, $u = \sum_{k=0}^\infty \langle u, \psi_{k,\varepsilon} \rangle_{\varrho_{\varepsilon,\infty}} \psi_{k,\varepsilon}$.
	The coefficients are given by for $k > 0$
	\begin{align*}
		\langle u, \psi_{k,\varepsilon} \rangle_{\varrho_{\varepsilon,\infty}} = \int (\varrho(x)-\varrho_{\varepsilon,\infty}(x)) \psi_{k,\varepsilon}(x) \dx  = \int \varrho(x) \psi_{k,\varepsilon}(x) \dx,
	\end{align*}
	and the first coefficient is $0$.
	Finally, computing the
	$\chi^2$-distance gives
	\begin{align*}
		(d_\varepsilon(\varrho))^2
		 & =
		\int \left | \sum_{k=1}^\infty \langle u, \psi_{k,\varepsilon} \rangle_{\varrho_{\varepsilon,\infty}} \psi_{k,\varepsilon}(x) \right |^2 \varrho_{\varepsilon,\infty}(x) \dx
		=
		\sum_{k=1}^\infty \langle u, \psi_{k,\varepsilon} \rangle_{\varrho_{\varepsilon,\infty}}^2
		\\
		 & =
		\sum_{k=1}^\infty \left ( \int \varrho(x) \psi_{k,\varepsilon}(x) \dx \right )^2.
	\end{align*}
\end{proof}

\begin{corollary}[Eigen-expansion of $\chi^2$-distance to equilibrium]\label{cor:separation}\hfill

	\noindent
	Assume that the operator \cref{eq:FP} has a discrete spectrum and as such there is a sequence of eigenvalues $0:=\lambda_{0,\varepsilon}<\lambda_{1,\varepsilon}<\ldots$ and a sequence of orthonormal eigenfunctions $\psi_{k,\varepsilon}$ spanning $L^2_{\varrho_{\varepsilon,\infty}}$.
	Let $\varrho$ be a density which solves the Fokker--Planck \cref{eq:OU:fokker} with initial data $\varrho_0 \in L^2_{\varrho_{\varepsilon,\infty}^{-1}}$, then
	\begin{align*}
		(d_\varepsilon(t))^2 := (d_\varepsilon(\varrho(\cdot,t)))^2 = \sum_{k=1}^\infty e^{- 2 \lambda_{k,\varepsilon} t} \left ( \int \varrho_0(x) \psi_{k,\varepsilon}(x) \dx \right )^2,\quad t\geq 0.
	\end{align*}
\end{corollary}
\begin{proof}
	The function $\varrho$ satisfies $\partial_t \varrho = \mathcal{L}^\ast_\varepsilon \varrho$, while $u=\frac{\varrho-\varrho_{\varepsilon,\infty}}{\varrho_{\varepsilon,\infty}}$ solves
	$\partial_t u = \mathcal{L}_\varepsilon u$. As such, expanding $u$ using the eigen-basis we see that the solution $u$ can be represented as
	\begin{align*}
		u = \sum_{k=0}^\infty e^{-\lambda_{k,\varepsilon} t} \langle u_0, \psi_{k,\varepsilon} \rangle_{\varrho_{\varepsilon,\infty}}\psi_{k,\varepsilon}, \quad \textrm{ where }\quad u_0:=\frac{\varrho_0-\varrho_{\varepsilon,\infty}}{\varrho_{\varepsilon,\infty}}.
	\end{align*}
	By \cref{lem:separation} the proof is completed.
\end{proof}

\section{The general case \texorpdfstring{$0<\gamma< 1$}{0 < gamma < 1}}\label{sec:menor}

In this section, we provide the main steps in comparison with the case $\gamma\in (1/3,1)$ for the proof of \cref{thm:main}.
We begin the appendix by stating the detailed structure of the functions $S_n'(x)$ defined in \cref{ec:S0S1Sk} of \cref{lem:WKBprime} for all $0<\gamma< 1$.
\begin{lemma}[General structure of $S_n'$]
	\label{lem:general-shape}
	\hfill

	\noindent
	Let $0 < \gamma < 1$, let $f(x)=|x|^\gamma$, let $g(x)=\lambda>0$, and let $(S_n)_{n\ge1}$ be defined by the recurrence in \cref{ec:S0S1Sk} of \cref{lem:WKBprime}. Then, for each $n\ge1$, there exist polynomials $A_{n,i}(\gamma)$, $1\le i\le n$, such that for $x\ge x^\ast$,
	\begin{equation} \label{eq:Sk:formula:explicit}
		S_n'(x) = \sum_{i=1}^{n} \lambda^i \, A_{n,i}(\gamma)\,|x|^{-((n+i-1)\gamma + (n-i))},
	\end{equation}
	and each $A_{n,i}(\gamma)$ has degree $n-i$ and sign $(-1)^{n-i}$. The leading coefficient is given by
	\begin{equation*}
		A_{n,n} = \frac{1}{n}\binom{2(n-1)}{n-1}.
	\end{equation*}
	Consequently, if $\psi_{k,\varepsilon}$ is the $k$-th non-constant eigenfunction of $-\mathcal{L}_\varepsilon$ with eigenvalue $\lambda_{k,\varepsilon} = \lambda_k \varepsilon^{\frac{\gamma-1}{1+\gamma}}$, and if $N := \big\lceil \frac{1+\gamma}{2\gamma} \big\rceil$, then
	\begin{equation}\label{eq:suma}
		\psi_{k,\varepsilon}(x) \approx \exp \biggl ( \sum_{i=1}^N \int_{x^\ast}^{x \varepsilon^{\frac{-1}{1+\gamma}}} S_i'(s )  \ds\biggr ),
	\end{equation}
	where $S_i'$ is given by \cref{eq:Sk:formula:explicit}. In the resonant case $N = \frac{1+\gamma}{2\gamma}\in \mathbb N$, the last term $i=N$ contributes a logarithmic correction.
\end{lemma}
\begin{proof}
	The proof follows by complete induction on $n$. For $n=1$, we have $S_1'(x)=\lambda |x|^{-\gamma}$ and $A_{1,1}=1$. Assume that \cref{eq:Sk:formula:explicit} holds for all indices strictly less than $n$.
	First consider the linear term in \cref{ec:S0S1Sk}. If a monomial in $S_{n-1}'$ has the form
	\begin{equation*}
		\lambda^i A_{n-1,i}(\gamma) |x|^{-((n+i-2)\gamma +(n-1-i))},
	\end{equation*}
	then differentiating once and dividing by $|x|^\gamma$ gives
	\begin{equation*}
		-\bigl((n+i-2)\gamma +(n-1-i)\bigr)\lambda^i A_{n-1,i}(\gamma)
		|x|^{-((n+i-1)\gamma +(n-i))},
	\end{equation*}
	which preserves the claimed exponent pattern. Moreover, the coefficient is again a polynomial in $\gamma$, now of degree at most $(n-1-i)+1=n-i$, and its sign is $(-1)^{n-i}$ because the derivative contributes one additional minus sign.
	It remains to analyze the quadratic term on the right-hand side of \cref{ec:S0S1Sk}, namely
	\begin{equation*}
		\frac{1}{|x|^\gamma}\sum_{j=1}^{n-1} S_j'(x)\,S_{n-j}'(x).
	\end{equation*}
	Expanding each $S_j'$ by the induction hypothesis, a typical product of a term indexed by $a$ in $S_j'$ and a term indexed by $b$ in $S_{n-j}'$ is proportional to
	\begin{equation*}
		\lambda^{a+b}A_{j,a}(\gamma)A_{n-j,b}(\gamma)
		|x|^{-((j+a-1)\gamma +(j-a)) - (((n-j)+b-1)\gamma +(n-j-b)) -\gamma}
		.
	\end{equation*}
	After simplification, this becomes
	\begin{equation*}
		\lambda^i
		|x|^{-((n+i-1)\gamma+(n-i))},
	\end{equation*}
	where $i:=a+b$. Since $A_{j,a}(\gamma)$ and $A_{n-j,b}(\gamma)$ have degrees $j-a$ and $(n-j)-b$ respectively, the product has degree $(j-a)+((n-j)-b)=n-i$. Likewise, its sign is
	\begin{equation*}
		(-1)^{j-a}(-1)^{(n-j)-b}=(-1)^{n-i}.
	\end{equation*}
	Hence, every contribution to the coefficient of $\lambda^i$ has the same sign and degree at most $n-i$. This yields \cref{eq:Sk:formula:explicit} for $n$, with coefficients $A_{n,i}(\gamma)$ that are polynomials in $\gamma$ of degree $n-i$ and sign $(-1)^{n-i}$.
	This completes the induction.

	Moreover, the top-order coefficient (corresponding to $i=n$) satisfies the recursion
	\begin{equation*}
		A_{n,n}=\sum_{j=1}^{n-1}A_{j,j}\,A_{n-j,n-j},\qquad A_{1,1}=1,
	\end{equation*}
	whose solution is given by the Catalan numbers
	\begin{equation}\label{eq:Akk_closed}
		A_{n,n}=\frac{1}{n}\binom{2(n-1)}{n-1}.
	\end{equation}
	The proof of \cref{eq:suma} follows similarly as the proof of \cref{thm:main:3prime}.
\end{proof}

In the sequel, we introduce an \textbf{ordering convention} that will be convenient to collect the growth of the eigenfunctions.
Whenever $\frac{1}{2n+1} < \gamma < \frac{1}{2n-1}$ for some $n \in \mathbb{N}$, then
\begin{equation*}
	n+1=N:=\lceil \tfrac{1+\gamma}{2\gamma} \rceil > \tfrac{1+\gamma}{2\gamma}.
\end{equation*}
We can then list the terms in strictly increasing order of the exponents in \cref{eq:Sk:formula:explicit}, $E_{j,i}=(j+i-1)\gamma+(j-i)$, by grouping pairs $(j,i)$ according to the diagonal $d:=j-i$ (from $d=0$ up to $d=n$), and within each diagonal by increasing $j$ (from $j=d+1$ to $j=n+1$). This yields a total order with no ties, and we can write
\begin{equation}\label{eq:indexz}
	\begin{aligned}
		\sum_{i=1}^{n+1} S_i'(x)
		 & =
		\sum_{j=1}^{n+1}\sum_{i=1}^{j}
		\lambda^{\,i}\,A_{j,i}(\gamma)\,
		|x|^{-((j+i-1)\gamma+(j-i))}
		\\
		 & =
		\sum_{d=0}^{n}\sum_{j=d+1}^{n+1}
		\lambda^{\,j-d}\,
		A_{j,j-d}(\gamma)\,
		|x|^{-((2j-d-1)\gamma+d)}.
	\end{aligned}
\end{equation}
Within each $d$-block, the numbers $(2j-d-1)\gamma+d$ increase with $j$ and we have
\begin{equation*}
	\begin{aligned}
		E_{\min}(d) & := (d+1)\gamma + d,                             \\
		E_{\max}(d) & := (2n+1 - d)\gamma + d,\qquad d = 0,1,\dots,n,
	\end{aligned}
\end{equation*}
from which we get
\begin{equation*}
	E_{\min}(d) > E_{\max}(d-1)\quad  \textrm{ for all } \quad d=1,\dots,n.
\end{equation*}
Hence, the blocks are ordered by increasing $d$.
The leading order block is $d=0$ with exponents of the form
\begin{equation*}
	-(2j-1)\gamma, \quad j=1,\dots,n+1,
\end{equation*}
which is precisely the sequence of exponents in \cref{thm:main:3prime} after integration.
Consequently, \cref{eq:indexz} implies that the eigenfunction $\psi_{k,\varepsilon}$ can be expressed as
\begin{equation*}
	\psi_{k,\varepsilon}(x)
	\approx
	\exp\biggl(\sum_{d=0}^{n}\sum_{j=d+1}^{n+1}
	\lambda_k^{\,j-d}\,A_{j,j-d}(\gamma)\,
	\frac{|x \varepsilon^{\frac{-1}{1+\gamma}}|^{1-((2j-d-1)\gamma+d)}}{1-((2j-d-1)\gamma+d)}
	\biggr),
\end{equation*}
where from now on we denote
\begin{equation*}
	\hat A_{j,j-d}(\gamma) := \frac{A_{j,j-d}(\gamma)}{1-((2j-d-1)\gamma+d)}.
\end{equation*}

Using the above notation, we can now state the generalization of \cref{lem:pointwise_intro,lem:time_intro} for the case $0<\gamma<1$.
\begin{lemma}[Growth of the Fourier coefficients for $0<\gamma<1$]\label{lem:pointwise_smallgamma}
	\hfill

	\noindent
	Let $\gamma\in(0,1)$, let $\varrho_{\varepsilon,0}$ be the initial density from \cref{eq:initialcond} with $\delta(\varepsilon)=\varepsilon^{\frac{1-\gamma}{1+\gamma}}$, let $\psi_{k,\varepsilon}$ be the $k$-th non-constant eigenfunction of $-\mathcal{L}_\varepsilon$ with eigenvalue $\lambda_{k,\varepsilon}=\lambda_k\varepsilon^{\frac{\gamma-1}{1+\gamma}}$, and let $N$ be as in \cref{lem:general-shape}.
	With the notation introduced after \cref{lem:general-shape}, define for $0\leq d\leq N-1$ and $d+1\leq j\leq N$
	\begin{equation}\label{eq:alpha_def}
		\alpha_{j,d}:=(2j-d-1)\gamma+d.
	\end{equation}
	Assume first that $x_0\neq 0$. In the non-resonant case, that is when $N>\frac{1+\gamma}{2\gamma}$, for all sufficiently small $\varepsilon>0$ one has
	\begin{equation*}
		\int \varrho_{\varepsilon,0}(x)\psi_{k,\varepsilon}(x)\,\dx
		\approx
		\exp\Bigg(\sum_{d=0}^{N-1}\sum_{j=d+1}^{N}
		\lambda_k^{\,j-d}\,\hat A_{j,j-d}(\gamma)\,
		\varepsilon^{\frac{\alpha_{j,d}-1}{1+\gamma}}
		|x_0|^{\,1-\alpha_{j,d}}
		\Bigg),
	\end{equation*}
	where the implied comparability constants are independent of $\varepsilon$. In the resonant case $N = \frac{1+\gamma}{2\gamma}$, the same asymptotic holds except that the term corresponding to $(j,d)=(N,0)$ is replaced by 
	\[
		\lambda_k^N A_{N,N}(\gamma)\left(\log|x_0| - \frac{1}{1+\gamma}\log\varepsilon\right).
	\]

	If $x_0=0$ and $\psi_k(0)=0$, then the Fourier coefficient vanishes. If $x_0=0$ and $\psi_k(0)\neq0$, then in the non-resonant case $N>\frac{1+\gamma}{2\gamma}$ one has
	\begin{equation}\label{eq:center-even}
		\left\lvert \int \varrho_{\varepsilon,0}(x)\psi_{k,\varepsilon}(x)\,\dx\right\rvert
		\approx
		\varepsilon^{\frac{\gamma(1-\gamma)}{1+\gamma}}
		\exp\!\Bigg(
		\sum_{d=0}^{N-1}\sum_{j=d+1}^{N}
		\lambda_k^{\,j-d}\,\hat A_{j,j-d}(\gamma)\,
		\varepsilon^{\frac{\gamma(\alpha_{j,d}-1)}{1+\gamma}}
		\Bigg).
	\end{equation}
	In the resonant case $N=\frac{1+\gamma}{2\gamma}$, the same asymptotic holds except that the term corresponding to $(j,d)=(N,0)$ is replaced by $-\lambda_k^{N}A_{N,N}(\gamma)\,\frac{\gamma}{1+\gamma}\log\varepsilon$.
\end{lemma}
The proof is analogous to the proof of \cref{lem:pointwise_intro}, using \cref{lem:general-shape,eq:indexz} in place of \cref{cor:asymptoticsprime}, and is left to the reader.
\begin{lemma}[Mixing times asymptotics for $0 < \gamma < 1$]\label{lem:time_smallgamma}
	\hfill

	\noindent
	Let $0 < \gamma < 1$, let $k\in\mathbb N$, let $\psi_{k,\varepsilon}$ be the $k$-th non-constant eigenfunction of $-\mathcal{L}_\varepsilon$ with eigenvalue $\lambda_{k,\varepsilon}=\lambda_k\varepsilon^{\frac{\gamma-1}{1+\gamma}}$, let $N$ be as in \cref{lem:general-shape}, and let $\alpha_{j,d}$ be as in \cref{eq:alpha_def}.
	If $x_0 \neq 0$, define
	\begin{equation*}
		S_{k}(x_0,\varepsilon)
		:=
		\sum_{d=0}^{N-1}\sum_{j=d+1}^{N}
		\lambda_k^{\,j-d}\, \hat A_{j,j-d}(\gamma)\,
		\varepsilon^{\frac{\alpha_{j,d}-1}{1+\gamma}}
		|x_0|^{\,1-\alpha_{j,d}},
	\end{equation*}
	where, in the resonant case $N=\frac{1+\gamma}{2\gamma}$, the term corresponding to $(j,d)=(N,0)$ is replaced by $\lambda_k^N A_{N,N}(\gamma)\left(\log|x_0| - \frac{1}{1+\gamma}\log\varepsilon\right)$.
	Let $t_\varepsilon$ be the (unique) time satisfying
	\begin{equation}\label{eq:teps_def}
		e^{-2\lambda_k \varepsilon^{\frac{\gamma-1}{1+\gamma}} t_\varepsilon}
		\Big(\!\int \varrho_{\varepsilon,0}(x)\psi_{k,\varepsilon}(x)\dx\!\Big)^2
		=1.
	\end{equation}
	Then for all sufficiently small $\varepsilon>0$ there exists a constant $C=C(\gamma,k)\geq 1$ such that
	\begin{equation}\label{eq:teps_bounds}
		\frac{\varepsilon^{\frac{1-\gamma}{1+\gamma}}}{\lambda_k}\bigl(S_{k}(x_0,\varepsilon)-\log C\bigr)
		\le
		t_\varepsilon
		\le
		\frac{\varepsilon^{\frac{1-\gamma}{1+\gamma}}}{\lambda_k}\bigl(S_{k}(x_0,\varepsilon)+\log C\bigr).
	\end{equation}

	If $x_0=0$ and $\psi_k(0)\neq 0$, define

	\begin{equation*}
		S_k^{\mathrm{cen}}(\varepsilon)
		:=
		\frac{\gamma(1-\gamma)}{1+\gamma}\log\varepsilon
		+
		\sum_{d=0}^{N-1}\sum_{j=d+1}^{N}
		\lambda_k^{\,j-d}\, \hat A_{j,j-d}(\gamma)\,
		\varepsilon^{\frac{\gamma(\alpha_{j,d}-1)}{1+\gamma}},
	\end{equation*}
	where, in the resonant case $N=\frac{1+\gamma}{2\gamma}$, the term corresponding to $(j,d)=(N,0)$ is replaced by $-\lambda_k^{N}A_{N,N}(\gamma)\,\frac{\gamma}{1+\gamma}\log\varepsilon$.
	Let $t_\varepsilon$ be the (unique) time satisfying \cref{eq:teps_def}. Then for all sufficiently small $\varepsilon>0$ there exists a constant $C=C(\gamma,k)\geq 1$ such that
	\begin{equation}\label{eq:teps_bounds_center}
		\frac{\varepsilon^{\frac{1-\gamma}{1+\gamma}}}{\lambda_k}\bigl(S_k^{\mathrm{cen}}(\varepsilon)-\log C\bigr)
		\le
		t_\varepsilon
		\le
		\frac{\varepsilon^{\frac{1-\gamma}{1+\gamma}}}{\lambda_k}\bigl(S_k^{\mathrm{cen}}(\varepsilon)+\log C\bigr).
	\end{equation}
\end{lemma}
The proof follows from \cref{lem:pointwise_smallgamma} by taking logarithms in the defining relation \cref{eq:teps_def}, exactly as in the proof of \cref{lem:time_intro}, and is left to the reader.

Before stating the main result for $0<\gamma<1/3$, we analyze the limit of $t_\varepsilon$ as $\varepsilon\to0$.
The exponents of $\varepsilon$ appearing in the expansion of $\varepsilon^{\frac{1-\gamma}{1+\gamma}} S_k(x_0,\varepsilon)$ is
\begin{equation*}
	\tfrac{1-\gamma}{1+\gamma}+\tfrac{\alpha_{j,d}-1}{1+\gamma} = \frac{\alpha_{j,d}-\gamma}{1+\gamma}
	= \frac{(2j-d-2)\gamma + d}{1+\gamma},
\end{equation*}
which is zero when $j=1$ and $d=0$. Otherwise,
\begin{equation*}
	(2j-d-2)\gamma + d \geq (2(d+1)-d-2)\gamma + d = d(1+\gamma) \geq 0,
\end{equation*}
which is strictly positive for all other pairs $(j,d)$. As such we have $\lim_{\varepsilon\to0} t_\varepsilon = \frac{|x_0|^{1-\gamma}}{1-\gamma}$. Which in particular implies that $\lim_{\varepsilon\to0} \lambda_{1,\varepsilon} t_\varepsilon = \infty$ for all $0<\gamma<1$.

If instead $x_0=0$ and $\psi_k(0)\neq 0$, we use \cref{eq:teps_bounds_center}. The exponents of $\varepsilon$ appearing in the expansion of $\varepsilon^{\frac{1-\gamma}{1+\gamma}}S_k^{\mathrm{cen}}(\varepsilon)$ are
\begin{equation*}
	\frac{1-\gamma}{1+\gamma}+\frac{\gamma(\alpha_{j,d}-1)}{1+\gamma}
	=
	\frac{(1-\gamma)^2 + \gamma(\alpha_{j,d}-\gamma)}{1+\gamma}.
\end{equation*}
Since $\alpha_{j,d}\geq \gamma$ with equality only for $(j,d)=(1,0)$, the smallest exponent is $\frac{(1-\gamma)^2}{1+\gamma}$, coming from the term $(j,d)=(1,0)$, where $\hat A_{1,1}(\gamma)=\frac{1}{1-\gamma}$. Consequently,
\begin{equation*}
	t_\varepsilon \approx \frac{1}{1-\gamma}\,\varepsilon^{\frac{(1-\gamma)^2}{1+\gamma}},
\end{equation*}
and in particular
\begin{equation*}
	\lim_{\varepsilon\to0} t_\varepsilon = 0,
	\qquad
	\lim_{\varepsilon\to0} \lambda_{k,\varepsilon} t_\varepsilon = \infty,
\end{equation*}
which is just the product condition.

\begin{theorem}[Window/profile cut-off for small exponents $\gamma<1/3$]\label{thm:small_gamma_cutoff}\hfill

	\noindent
	Let $0<\gamma<\tfrac{1}{3}$, let $k\in\mathbb{N}$, and for $1\le j\le k$ let $\psi_{j,\varepsilon}$ be the $j$-th non-constant eigenfunction of $-\mathcal L_\varepsilon$ with eigenvalue $\lambda_{j,\varepsilon}=\lambda_j\varepsilon^{\frac{\gamma-1}{1+\gamma}}$. Let the initial density be as in \cref{eq:initialcond} with $\delta(\varepsilon)=\varepsilon^{\frac{1-\gamma}{1+\gamma}}$.
	For $x_0\neq0$, with the notation of \cref{lem:pointwise_smallgamma,lem:time_smallgamma} let $t_\varepsilon$ be the unique time defined by \cref{eq:teps_def} and set $w_\varepsilon=\varepsilon^{\frac{1-\gamma}{1+\gamma}}$. Then for any fixed $r\in\mathbb{R}$ it follows that
	\begin{equation*}
		\lim_{\varepsilon\to0} d_{k,\varepsilon}(t_\varepsilon+r w_\varepsilon,x_0)
		= e^{-\lambda_k r}.
	\end{equation*}
	In particular, the truncated $\chi^2$--distance exhibits window cut-off (indeed profile cut-off) with cut-off time $t_\varepsilon$, window $w_\varepsilon$ and profile function $\mathfrak p(r)=e^{-\lambda_k r}$, $r\in \mathbb{R}$.

	For $x_0=0$, let
	\begin{equation*}
		E_k:=\{1\leq j\leq k:\ \psi_j(0)\neq 0\}.
	\end{equation*}
	If $E_k=\emptyset$, then $d_{k,\varepsilon}(t,0)=0$ for all $t\geq 0$. Otherwise, letting $m:=\max E_k$, the centered initial datum exhibits window cut-off (indeed profile cut-off) with cut-off time determined by the condition
	\begin{equation*}
		e^{-2\lambda_m \varepsilon^{\frac{\gamma-1}{1+\gamma}} t_\varepsilon}
		\left(\int \varrho_{\varepsilon,0}(x)\psi_{m,\varepsilon}(x)\,\dx\right)^2=1,
	\end{equation*}
	the same time window $w_\varepsilon:=\varepsilon^{\frac{1-\gamma}{1+\gamma}}$, and profile function $\mathfrak p(r)=e^{-\lambda_m r}$.
\end{theorem}

\begin{proof}
	As in \cref{subsection:proofs}, write the vector entries
	\begin{equation*}
		T_{i,\varepsilon}:=
		e^{-2\lambda_i \varepsilon^{\frac{\gamma-1}{1+\gamma}} t_\varepsilon}
		\Bigl( \int \varrho_{\varepsilon,0}(x)\psi_{i,\varepsilon}(x)\,\dx \Bigr)^2,
		\qquad i=1,\dots,k.
	\end{equation*}
	By \cref{lem:pointwise_smallgamma} the squared Fourier coefficient admits the WKB expansion
	\begin{equation*}
		\Big( \int \varrho_{\varepsilon,0}\psi_{i,\varepsilon} \dx\Big)^2
		\approx \exp\Big( \sum_{d=0}^{N-1}\sum_{j=d+1}^{N} 2\lambda_i^{\,j-d}\, \hat A_{j,j-d}(\gamma)\,
		\varepsilon^{\frac{\alpha_{j,d}-1}{1+\gamma}}
		|x_0|^{\,1-\alpha_{j,d}} \Big),
	\end{equation*}
	with $N=\lceil(1+\gamma)/(2\gamma)\rceil$ and $\alpha_{j,d}$ as in \cref{eq:alpha_def}. Substitute the bounds of $t_\varepsilon$ from \cref{lem:time_smallgamma} (equivalently \cref{eq:teps_bounds}) and combine the resulting exponents, rearranging yields, up to an additive constant $C$ independent of $\varepsilon$,
	\begin{equation}\label{eq:single-sum}
		\begin{aligned}
			 & -2\lambda_i \varepsilon^{\frac{\gamma-1}{1+\gamma}} t_\varepsilon
			+\sum_{d=0}^{N-1}\sum_{j=d+1}^{N}
			2\lambda_i^{\,j-d} \hat A_{j,j-d}(\gamma)\,
			\varepsilon^{\frac{\alpha_{j,d}-1}{1+\gamma}}
			|x_0|^{\,1-\alpha_{j,d}}                                             \\
			 & \qquad\le
			\sum_{d=0}^{N-1}\sum_{j=d+1}^{N}
			\biggl[
			\Big(\lambda_i^{\,j-d}-\frac{\lambda_i}{\lambda_k}\lambda_k^{\,j-d}\Big)
			2 \hat A_{j,j-d}(\gamma)\,
			\varepsilon^{\frac{\alpha_{j,d}-1}{1+\gamma}}
			|x_0|^{\,1-\alpha_{j,d}}
			\biggr] + C.
		\end{aligned}
	\end{equation}
	Inspect the leading terms in the bracket of \cref{eq:single-sum}. Whenever $j-d=1$, that is, $j=d+1$, the coefficient vanishes since
	\begin{equation*}
		\lambda_i^{\,j-d}-\frac{\lambda_i}{\lambda_k}\lambda_k^{\,j-d}
		=
		\lambda_i-\frac{\lambda_i}{\lambda_k}\lambda_k
		=0.
	\end{equation*}
	Hence, all terms with $j-d=1$ cancel. Among the remaining terms, the smallest exponent of $\varepsilon$ is attained for $(j,d)=(2,0)$, for which
	\begin{equation*}
		\alpha_{2,0}-1 = 3\gamma-1.
	\end{equation*}
	Since $\gamma<\tfrac{1}{3}$, this exponent is negative, and the corresponding coefficient
	\begin{equation*}
		2\Big(\lambda_i^{2-0}-\frac{\lambda_i}{\lambda_k}\lambda_k^{2-0}\Big)\hat A_{2,2}
		=2\lambda_i\big(\lambda_i-\lambda_k\big)\hat A_{2,2}
	\end{equation*}
	is strictly negative for $i<k$ (since $\lambda_i<\lambda_k$) and
	\begin{equation*}
		\hat A_{2,2} = \frac{A_{2,2}}{1-3\gamma} >0,
	\end{equation*}
	by using \cref{eq:Akk_closed} and the fact that $0 < \gamma < \tfrac{1}{3}$.
	Consequently, the bracket in \cref{eq:single-sum} is dominated by a negative term which blows up to $-\infty$ as $\varepsilon\to0$; hence $T_{i,\varepsilon}\to0$ for every $i<k$. By construction (see \cref{eq:teps_def}) the $i=k$ component satisfies $T_{k,\varepsilon}=1$ for all $\varepsilon$.

	Therefore, the vector $V_{k,\varepsilon}=(T_{1,\varepsilon},\dots,T_{k,\varepsilon})$ has a single non-vanishing entry in the limit, namely the $k$-th entry equal to $1$, while the window factor
	$W_k=\{e^{-2\lambda_i r}\}_{i=1}^k$ is independent of $\varepsilon$. It follows that for any fixed $r\in\mathbb{R}$
	\begin{equation*}
		\lim_{\varepsilon\to0} d_{k,\varepsilon}(t_\varepsilon+r w_\varepsilon,x_0)
		=\sqrt{\lim_{\varepsilon\to0} V_{k,\varepsilon}\cdot W_k}
		=e^{-\lambda_k r}.
	\end{equation*}
	This proves window (indeed profile) cut-off with profile $\mathfrak p(r)=e^{-\lambda_k r}$ for $x_0\neq 0$.

	The centered case $x_0=0$ is treated as in the centered part of the proof of \cref{thm:main}. If $E_k=\emptyset$, then every Fourier coefficient vanishes by parity and $d_{k,\varepsilon}(t,0)=0$ for all $t\geq 0$. Otherwise, let $m:=\max E_k$ and define $t_\varepsilon$ through the $m$-th mode. Using \cref{eq:center-even,eq:teps_bounds_center}, for every $j\in E_k$ with $j<m$ we obtain
	\begin{align*}
		 & -2\lambda_j \varepsilon^{\frac{\gamma-1}{1+\gamma}} t_\varepsilon
		+2\frac{\gamma(1-\gamma)}{1+\gamma}\log\varepsilon
		+2\sum_{d=0}^{N-1}\sum_{l=d+1}^{N}
		\lambda_j^{\,l-d}\,\hat A_{l,l-d}(\gamma)\,
		\varepsilon^{\frac{\gamma(\alpha_{l,d}-1)}{1+\gamma}}
		\\
		 & \qquad\le
		2\frac{\gamma(1-\gamma)}{1+\gamma}\Bigl(1-\frac{\lambda_j}{\lambda_m}\Bigr)\log\varepsilon
		+R_{j,\varepsilon},
	\end{align*}
	where $R_{j,\varepsilon}$ is bounded above for small $\varepsilon$; indeed, after the cancellation of the terms with $l-d=1$, its first non-vanishing contribution is the $(l,d)=(2,0)$ term, whose exponent $\frac{\gamma(3\gamma-1)}{1+\gamma}$ is negative and whose coefficient is also negative since $\lambda_j<\lambda_m$. Therefore, the whole exponent tends to $-\infty$ as $\varepsilon\to0$, and hence
	\begin{equation*}
		e^{-2\lambda_j \varepsilon^{\frac{\gamma-1}{1+\gamma}} t_\varepsilon}
		\left(\int \varrho_{\varepsilon,0}(x)\psi_{j,\varepsilon}(x)\,\dx\right)^2\to 0
		\qquad\text{for every } j<m.
	\end{equation*}
	By construction the $m$-th mode is normalized to equal $1$, and therefore
	\begin{equation*}
		d_{k,\varepsilon}(t_\varepsilon+r w_\varepsilon,0)\to e^{-\lambda_m r}.
	\end{equation*}
\end{proof}

\section{Scaling of the generator \texorpdfstring{$\mathcal{L}_\varepsilon$}{Lepsilon}}\label{ap:tools}
In this section, for completeness we give the proof of
\cref{lem:g:scaling}.
\begin{proof}[Proof of \cref{lem:g:scaling}]
	Let $\rho:=\varepsilon^{\frac{1}{1+\gamma}}$ and set $u_\varepsilon(x):=u(x/\rho)$. Writing $y:=x/\rho$, the chain rule gives
	\begin{equation*}
		u_\varepsilon'(x)=\rho^{-1}u'(y),
		\qquad
		u_\varepsilon''(x)=\rho^{-2}u''(y).
	\end{equation*}
	Since $V'(\rho y)=\rho^\gamma V'(y)$ and $\varepsilon=\rho^{1+\gamma}$, we obtain
	\begin{equation*}
		\mathcal{L}_\varepsilon u_\varepsilon(x)
		=\varepsilon u_\varepsilon''(x)-V'(x)u_\varepsilon'(x)
		=\rho^{\gamma-1}\bigl(u''(y)-V'(y)u'(y)\bigr)
		=\rho^{\gamma-1}(\mathcal{L}_1u)(y).
	\end{equation*}
	Hence, if $\mathcal{L}_1u=f$, then
	\begin{equation*}
		\mathcal{L}_\varepsilon u_\varepsilon(x)
		=\varepsilon^{\frac{\gamma-1}{\gamma+1}}f\left(x\varepsilon^{-\frac{1}{1+\gamma}}\right).
	\end{equation*}
	If instead $-\mathcal{L}_1u=\lambda u$, the same identity gives
	\begin{equation*}
		-\mathcal{L}_\varepsilon u_\varepsilon
		=\lambda\varepsilon^{\frac{\gamma-1}{\gamma+1}}u_\varepsilon.
	\end{equation*}
	Finally, using the change of variables $x=\rho y$ and $V(\rho y)=\rho^{1+\gamma}V(y)=\varepsilon V(y)$,
	\begin{equation*}
		\|u_\varepsilon\|_{L^2(C_\varepsilon e^{-V/\varepsilon})}^2
		=C_\varepsilon\rho\int_{\mathbb{R}} |u(y)|^2 e^{-V(y)}\,\dy.
	\end{equation*}
	Moreover,
	\begin{equation*}
		C_\varepsilon^{-1}=\int_{\mathbb{R}} e^{-V(x)/\varepsilon}\,\dx
		=\rho\int_{\mathbb{R}} e^{-V(y)}\,\dy
		=\rho C_1^{-1},
	\end{equation*}
	so $C_\varepsilon\rho=C_1$. Therefore,
	\begin{equation*}
		\|u_\varepsilon\|_{L^2(C_\varepsilon e^{-V/\varepsilon})}^2
		=\|u\|_{L^2(C_1 e^{-V})}^2,
	\end{equation*}
	which concludes the proof.
\end{proof}

\section*{{Statements and declarations}}

\noindent
\textbf{Acknowledgments.} Gerardo Barrera would like to express his gratitude to the Center for Mathematical Analysis, Geometry and Dynamical Systems CAMGSD and Instituto Superior T\'ecnico (IST) Lisbon, Portugal
for all the facilities used along the realization of this work.
Benny Avelin would also like to thank the CAMGSD and IST for support and hospitality where partial work on this paper was undertaken.

\noindent
\textbf{Funding.}
The research of Gerardo Barrera is supported by Horizon Europe Marie Sk\l{}odowska-Curie Actions Staff Exchanges (project no.\ 101183168 -- LiBERA). Also, the research of Gerardo Barrera is partially funded by Funda\c{c}\~ao para a Ci\^encia e Tecnologia (FCT), Portugal, through grant FCT/Portugal  project no. UID/04459/2025 with DOI identifier
10-54499/UID/04459/2025.
The research of Benny~Avelin has partially been supported by [Swedish Research Council dnr: 2019-04098].

\noindent
\textbf{Availability of data and material.}
Data sharing not applicable to this article as no data-sets were generated or analyzed during the current study.

\noindent
\textbf{Conflict of interests.} The authors declare that they have no conflict of interest.

\noindent
\textbf{Ethical approval.} Not applicable.

\noindent
\textbf{Authors' contributions.}
All authors have contributed equally to the manuscript.

\noindent
\textbf{Disclaimer.}
Funded by the European Union. Views and opinions expressed are however those of the author(s) only and do not necessarily reflect those of the European Union or the European Education and Culture Executive Agency (EACEA). Neither the European Union nor EACEA can be held responsible for them.


\begin{thebibliography}{50}

	%A
	\bibitem{AD}
	Aldous, D. \& Diaconis, P.:
	Shuffling cards and stopping times.
	\textit{Amer. Math. Monthly} \textbf{93}, no. 5, (1986), 333--348.
	%\url{https://doi.org/10.2307/2323590}
	\MR{0841111}

	\bibitem{AKL}
	Arnaudon, M., Coulibaly-Pasquier, K. \& Miclo, L.:
	On the separation cut-off phenomenon for Brownian motions on high dimensional spheres.
	\textit{Bernoulli } \textbf{30}, no. 2, (2024), 1007--1028.
	%%\url{https://doi.org/10.3150/23-BEJ1622}
	\MR{4699543}

	\bibitem{Avelin}
	Avelin, B. \&  Karlsson, A.:
	Deep limits and a cut-off phenomenon for neural networks.
	\textit{J. Mach. Learn. Res.} \textbf{23}, no. 191, 29 pp., (2022).
	%%\url{http://jmlr.org/papers/v23/21-0431.html}
	\MR{439458}

	%B
	\bibitem{BGL}
	Bakry, D., Gentil, I. \& Ledoux, M.:
	\textit{Analysis and geometry of Markov diffusion operators}.
	Springer, Cham, (2014).
	%%\url{https://doi.org/10.1007/978-3-319-00227-9}
	\MR{3155209}

	\bibitem{BCJ}
	Barrera, G., Da Costa, C. \& Jara, M.:
	Gradual convergence for Langevin dynamics on a degenerate potential.
	\textit{Stochastic Process. Appl.} \textbf{184}, Paper No. 104601, 28 pp., (2025).
	%%\url{https://doi.org/10.1016/j.spa.2025.104601}
	\MR{4864978}

	\bibitem{BARESQ}
	Barrera, G. \& Esquivel, L.:
	Profile cut-off phenomenon for the ergodic Feller root process.
	\textit{Stochastic Process. Appl.} \textbf{183}, Paper No. 104587, 27 pp., (2025).
	%%\url{https://doi.org/10.1007/s10955-024-03308-6}
	\MR{4791608}

	\bibitem{BHPPshell}
	Barrera, G., H\"ogele, M.A., Pardo, J.C. \& Pavlyukevich, I.:
	Cutoff ergodicity bounds in Wasserstein distance for a viscous energy shell model with L\'evy noise.
	\textit{J. Stat. Phys.} \textbf{191}, no. 9, Paper No. 105, 24 pp., (2024).
	%%\url{https://doi.org/10.1016/j.spa.2025.104587}
	\MR{4860944}

	\bibitem{BJ}
	Barrera, G. \& Jara, M.:
	Abrupt convergence of stochastic small perturbations of one dimensional dynamical systems.
	\textit{J. Stat. Phys.} \textbf{163}, no. 1, (2016),
	113--138.
	%%\url{https://doi.org/10.1007/s10955-016-1468-1}
	\MR{3472096}

	\bibitem{BJ1}
	Barrera, G. \& Jara, M.:
	Thermalisation for small random perturbations of dynamical systems.
	\textit{Ann. Appl. Probab.} \textbf{30}, no. 3, (2020), 1164--1208.
	%%\url{https://doi.org/10.1214/19-AAP1526}
	\MR{4133371}

	\bibitem{BBF1}
	Barrera, J., Bertoncini, O. \& Fern\'andez, R.:
	Abrupt convergence and escape behavior for birth and death chains.
	\textit{J. Stat. Phys.} \textbf{137}, no. 4, (2009),
	595--623.
	%%\url{https://doi.org/10.1007/s10955-009-9861-7}
	\MR{2565098}

	\bibitem{BY} Barrera, J. \& Ycart, B.:
	Bounds for left and right window cutoffs.
	\textit{ALEA Lat. Am. J. Probab. Math. Stat.} \textbf{11}, no. 2, (2014), 445--458.
	%%\url{https://alea.impa.br/articles/v11/11-19.pdf}
	\MR{3265085}

	\bibitem{BHP}
	Basu, R., Hermon, J. \& Peres, Y.:
	Characterization of cutoff for reversible Markov chains.
	\textit{Ann. Probab.} \textbf{45}, no. 3, (2017), 1448--1487.
	%%\url{https://doi.org/10.1214/16-AOP1090}
	\MR{3650406}

	\bibitem{BOK10}
	Bayati, B., Owahi, H. \& Koumoutsakos, P.:
	A cutoff phenomenon in accelerated stochastic simulations of chemical kinetics via flow averaging (FLAVOR-SSA).
	\textit{J. Chem. Phys.} \textbf{133}, no. 24, Paper No. 244117, 6 pp., (2010). %%\url{https://doi.org/10.1063/1.3518419}

	\bibitem{Bogachev}
	Bogachev, V., Krylov, N.,
	R\"ockner, M. \& Shaposhnikov, S.: \textit{Fokker--Planck--Kolmogorov equations}.
	Mathematical Surveys and Monographs \textbf{207}. American Mathematical Society, Providence, RI, (2015).
	%%\url{https://doi.org/10.1090/surv/207}
	\MR{3443169}

	\bibitem{Bolley}
	Bolley, F., Gentil, I. \& Guillin, A.:
	Convergence to equilibrium in Wasserstein distance for Fokker--Planck equations.
	\textit{J. Funct. Anal.} \textbf{263}, no. 8, (2012), 2430--2457.
	%%\url{https://doi.org/10.1016/j.jfa.2012.07.007}
	\MR{2964689}

	\bibitem{BorMan}
	Bordenave, C. \& Mathien, J.:
	Cutoff for geodesic paths on hyperbolic manifolds.
	\textit{Commun. Math. Phys.} \textbf{407}, no. 111,  27 pp., (2026).
	%%\url{https://doi.org/10.1007/s00220-026-05607-3}

	%C
	\bibitem{ChenGY}
	Chen G.Y.:
	Mixing reversible Markov chains in the max-$\ell_2$-distance.
	\textit{ALEA, Lat. Am. J. Probab. Math. Stat.} \textbf{21}, no. 2, (2024), 1727--1767.
	%%\url{https://doi.org/10.30757/ALEA.v21-65}
	\MR{4801341}

	\bibitem{ChenHsuSheu}
	Chen, G.-Y., Hsu, J.-M. \&  Sheu, Y.-C.:
	The $L_2$-cutoffs for reversible Markov chains
	\textit{Ann. Appl. Probab.} \textbf{27}, no. 4, (2017), 2305--2341.
	%%\url{https://doi.org/10.1214/16-AAP1260}
	\MR{3693527}

	\bibitem{ChenSCL}
	Chen G.Y. \& Saloff-Coste, L.:
	Spectral computations for birth and death chains.
	\textit{Stochastic Process. Appl.} \textbf{124}, no. 1, 848--882, (2014).
	%%\url{https://doi.org/10.1016/j.spa.2013.10.002}
	\MR{3131316}

	\bibitem{CSC}
	Chen, G.Y. \& Saloff-Coste, L.:
	The cutoff phenomenon for ergodic Markov processes.
	\textit{Electron. J. Probab.} \textbf{13}, no. 3, (2008), 26--78.
	%%\url{https://doi.org/10.1214/EJP.v13-474}
	\MR{2375599}

	\bibitem{ChenGYSaloff2024}
	Chen G.Y. \& Saloff-Coste, L.:
	The $L_2$-cutoff for reversible Markov processes.
	\textit{J. Funct. Anal.} \textbf{258}, no. 7, (2010), 2246--2315.
	%%\url{https://doi.org/10.1016/j.jfa.2009.10.017}
	\MR{2584746}

	%D
	\bibitem{DI} Diaconis, P.:
	The cutoff phenomenon in finite Markov chains.
	\textit{Proc. Nat. Acad. Sci. U.S.A.} \textbf{93}, no. 4, (1996), 1659--1664.
	%%\url{https://doi.org/10.1073/pnas.93.4.1659}
	\MR{1374011}

	%E

	%F

	%G
	\bibitem{MHR}
	Gradinaru, M., Herrmann, S. \& Roynette, B.:
	A singular large deviations phenomenon.
	\textit{Ann. Inst. H. Poincar\'e Probab. Statist.} \textbf{37} no. 5, 555--580, (2001).
	%%\url{https://doi.org/10.1016/S0246-0203(01)01075-5}
	\MR{1851715}

	%H 
	\bibitem{Hopf}
	Hopf, E.:
	A remark on linear elliptic differential equations of second order.
	\textit{Proc. Amer. Math. Soc.} \textbf{3}, no. 5, 791--793, (1952).
	%%\url{https://doi.org/10.1090/S0002-9939-1952-0050126-X}
	\MR{0050126}

	%I

	%J
	\bibitem{Jac}
	Jacquot, S.:
	Comportement asymptotique de la seconde valeur propre des processus de Kolmogorov. \textit{J. Multivariate Anal.} \textbf{40}, no. 2, 335--347, (1992).  %%\url{https://doi.org/10.1016/0047-259X(92)90030-J}
	\MR{1150617}

	\bibitem{Ji}
	Ji, M., Shen, Z. \& Yi, Y.:
	Convergence to equilibrium in Fokker--Planck equations.
	\textit{J. Dyn. Diff. Equat.} \textbf{31}, no. 3, (2019), 1591--1615.
	%%\url{https://doi.org/10.1007/s10884-018-9705-8}
	\MR{3992083}

	%K
	\bibitem{Kastoryano12}
	Kastoryano, M.,  Reeb, D. \& Wolf, M.:
	A cutoff phenomenon for quantum Markov chains.
	\textit{J. Phys. A} \textbf{45}, no. 7, Paper No. 075307, 16 pp., (2012). %%\url{https://doi.org/10.1088/1751-8113/45/7/075307}
	\MR{2881085}

	\bibitem{KU}
	Kulik, A.:
	\textit{Ergodic behavior of Markov processes.
		With applications to limit theorems}. De Gruyter Studies in Mathematics \textbf{67}. De Gruyter, Berlin, (2018).
	%%\url{https://doi.org/10.1515/9783110458930}
	\MR{3791835}

	%L
	\bibitem{booklelievreetalt}
	Lelièvre, T., Rousset, M. \& Stoltz, G.:
	\textit{Free energy computations. A mathematical perspective}.
	Imperial College Press, London, (2010).
	%%\url{https://doi.org/10.1142/p579}
	\MR{2681239}

	\bibitem{Levin}
	Levin, D., Peres, Y. \& Wilmer, E.:
	\textit{Markov chains and mixing times}.
	With a chapter by James G. Propp and D. Wilson.
	Amer. Math. Soc., Providence, RI, (2009).
	%%\url{https://doi.org/10.1090/mbk/058}
	\MR{2466937}

	%M
	\bibitem{Maobook}
	Mao, X.:
	\textit{Stochastic differential equations and applications}.
	Second edition. Horwood Publishing Limited, Chichester, (2008).
	%%\url{https://doi.org/10.1533/9780857099402}
	\MR{2380366}

	\bibitem{Murray}
	Murray, R. \& Pego, R.:
	Cutoff estimates for the linearized Becker--D\"oring equations.
	\textit{Commun. Math.
		Sci.} \textbf{15}, no. 6,
	1685--1702, (2016). %%\url{https://doi.org/10.4310/CMS.2017.v15.n6.a10}
	\MR{3668953}

	%N

	%O
	\bibitem{Oh}
	Oh, S. \& Kais, S.:
	Cutoff phenomenon and entropic uncertainty for random quantum circuits.
	\textit{Electron. Struct.} \textbf{5},  no. 3,  7 pp., (2023).
	%%\url{https://doi.org/10.1088/2516-1075/acf2d3}

	\bibitem{Olver}
	Olver, F.W.J.:
	\textit{Asymptotics and special functions}. Reprint of the 1974 original Academic Press, New York, AKP Classics, A K Peters, Wellesley, MA, (1997).    %%\url{https://doi.org/10.1201/9781439864548}
	\MR{1429619}

	%P
	\bibitem{pages}
	Pag\`es, G. \& Panloup, F.:
	Ergodic approximation of the distribution of a stationary diffusion: rate of convergence.
	\textit{Ann. Appl. Probab.} \textbf{22} no. 3, (2012), 1059--1100.
	%%\url{https://doi.org/10.1214/11-AAP779}
	\MR{2977986}

	%Q

	%R
	\bibitem{Royer2007}
	Royer, G.:
	\textit{An initiation to logarithmic Sobolev inequalities}.
	American Mathematical Society, Providence, RI, (2007).
	\MR{2352327}

	%S
	\bibitem{Salez}
	Salez, J.:
	Cutoff for non-negatively curved diffusions.
    \textit{J. Eur. Math. Soc.} \textbf{26} no. 11, (2024), 4375--4392.

	\bibitem{SVbook}
	Stroock, D. \& Varadhan, S.R.S.:
	\textit{Multidimensional diffusion processes}.
	Classics Math. Springer--Verlag, Berlin, (2006).
	%%\url{https://doi.org/10.1007/3-540-28999-2}
	\MR{2190038}

	%T
	\bibitem{Titchmarsh}
	Titchmarsh, E.C.:
	\textit{Eigenfunction expansions associated with second-order differential equations}. Vol. 2. Oxford, Clarendon Press, (1958).
	\MR{0094551}

	%U

	%V
	\bibitem{Vernier20}
	Vernier, E.:
	Mixing times and cutoffs in open quadratic fermionic systems.
	\textit{SciPost Phys.} \textbf{9} no. 049, 1--30, (2020).    %%\url{https://doi.org/10.21468/SciPostPhys.9.4.049}
	\MR{4216481}

	\bibitem{Vi2009}
	Villani, C.:
	\textit{Optimal transport. Old and new}.
	Grundlehren der mathematischen Wissenschaften \textbf{338}. Springer--Verlag Berlin, (2009).
	%\url{https://doi.org/10.1007/978-3-540-71050-9}
	\MR{2459454}

	%W

	%X

	%Y

	%Z
\end{thebibliography}
\end{document}